\documentclass[11pt]{article}
\usepackage{IEEEtrantools}
\usepackage{nccmath}
\usepackage{mathtools}
\usepackage{physics}
\usepackage{latexsym}
\usepackage{amssymb}
\usepackage{amsmath,amsthm}
\usepackage{mathrsfs}
\usepackage{enumerate}
\usepackage{dsfont}
\usepackage{hyperref}
\usepackage{graphicx}
\usepackage{epstopdf}
\usepackage{tikz}
\usepackage[T1]{fontenc}

\usetikzlibrary{shapes,arrows.meta,positioning}
\usepackage{cleveref}
\usepackage{tikz-cd}
\usepackage{psfrag}

\usepackage{multirow}
\usetikzlibrary{arrows,decorations.markings}
\usepackage{bookmark}
\hypersetup{
     colorlinks   = true,
     citecolor    = blue
}

\newtheorem{dfn}{Definition}[section]
\newtheorem{thm}[dfn]{Theorem}
\newtheorem{lmma}[dfn]{Lemma}
\newtheorem{ppsn}[dfn]{Proposition}
\newtheorem{crlre}[dfn]{Corollary}

\newtheorem{rmrk}[dfn]{Remark}
\newtheorem{notation}[dfn]{Notation}

\newcommand{\la}{{\lambda}}
\newcommand{\bbc}{\mathbb{C}}
\newcommand{\bbz}{\mathbb{Z}}

\newcommand{\sk}{\mathcal{K}}

\hypersetup{linkcolor=blue}

\providecommand{\Tref}[1]{Theorem~\ref{#1}}
\providecommand{\Lref}[1]{Lemma~\ref{#1}}
\providecommand{\Pref}[1]{Proposition~\ref{#1}}
\providecommand{\Yref}[1]{Corollary~\ref{#1}}

\providecommand{\Eref}[1]{(\ref{#1})}

\let\oldthebibliography\thebibliography
\renewcommand{\thebibliography}[1]{
  \oldthebibliography{#1}
  \setlength{\itemsep}{0pt}
  \setlength{\parsep}{0pt}
  \setlength{\parskip}{0pt}
}

\begin{document}

\tikzset{->-/.style={decoration={
  markings,
  mark=at position #1 with {\arrow{>}}},postaction={decorate}}}
  \tikzset{-<-/.style={decoration={
  markings,
  mark=at position #1 with {\arrow{<}}},postaction={decorate}}}

\author{{\scshape Vibhor Bhatt,\,\,Satyajit Guin,\,\,Palash Sharma}}
\title{Gromov–Hausdorff Convergence of Spectral Truncations for Noncommutative Tori}

\date{}
\maketitle

\begingroup
\renewcommand{\thefootnote}{}
\begin{NoHyper}
\footnotetext{%

}
\end{NoHyper}
\endgroup

\begin{abstract}
We prove the Gromov--Hausdorff convergence of the state spaces, equipped with Connes' distance function, of the operator systems associated with spectral truncations for noncommutative $n$-tori, $n\ge 2$. The action of the $n$-torus on the noncommutative $n$-torus, together with an appropriate use of Bochner integrals, enables us to construct an $C^1$-approximate order isomorphism that establishes the convergence. The earlier convergence result for the classical case of tori \cite{LvS} follows as a special case of our result. We conclude by studying the $C^*$-envelopes and propagation numbers of the truncated operator systems and show that the structural properties of the truncated operator systems for tori persist under deformation.
\end{abstract}

%%%%%%%%%%%%%%%%%%%%%%%%%%%%%%%%%%%%%%%%%%%%
%%%%%%%%%%%%%%%%%%%%%%%%%%%%%%%%%%%%%%%%%%%%

\section{Introduction}

Connes and van Suijlekom \cite{CS1} introduced the notion of \emph{spectral truncations} in noncommutative geometry, motivated by the observation that, in many physical situations, only a finite portion of the spectrum is accessible. One may therefore seek to understand the underlying system by retaining only the spectral data below a prescribed cutoff and studying the resulting finite-dimensional approximation. The authors in \cite{CS2} have pointed out that our understanding of the geometry of spacetime is constrained by the resolution of our measuring devices; consequently, we can determine the underlying metric space only up to a finite resolution. This leads to the guiding question of whether, by considering suitable limiting procedures as the truncation is progressively relaxed, one can recover geometric information from these truncated models.

We begin with a spectral triple $(A,H,D)$ for a unital $C^*$ algebra $\mathscr{A}$ and obtain the truncated systems $\{(P_{\Lambda}AP_{\Lambda},P_{\Lambda}HP_{\Lambda},P_{\Lambda}DP_{\Lambda})\}_{\Lambda\ge 0}$ by the spectral projections $P_{\Lambda}$ of $D$ onto the eigenspaces with eigenvalues of modulus $\le\Lambda$. These truncated systems are finite-dimensional operator-system spectral triples and are called spectral truncations of $(A,H,D)$. Suppose that Connes' distance function \cite{C1,C2}, given by
\[
d_A(\phi,\psi):=\sup_{a\in A}\{|\phi(a)-\psi(a)|:\norm{[D,a]}\le 1\},
\]
induces the weak $\ast$-topology on the state space $\mathcal{S}(\mathscr{A})$; equivalently, the given spectral triple induces a compact quantum metric space in the sense of Rieffel \cite{Rie1,Rie2}. If $\,\ker([P_{\Lambda}DP_{\Lambda},\,.\,])$ consists only of scalars for every $\Lambda\ge 0$, then the finite-dimensionality of the truncated operator systems ensures that the corresponding Connes' distance function, with $A$ replaced by $P_\Lambda AP_\Lambda$ and $D$ replaced by $P_\Lambda DP_\Lambda$, induces the weak $\ast$-topology on the state spaces $\mathcal{S}(P_{\Lambda}\mathscr{A}P_{\Lambda})$ for every $\Lambda\ge 0$. With this formulation, the spectral truncation problem asks whether the family $\{\mathcal{S}(P_{\Lambda}\mathscr{A}P_{\Lambda})\}_{\Lambda}$ converges to $\mathcal{S}(\mathscr{A})$ in the Gromov--Hausdorff distance as $\Lambda\to\infty$. As mentioned in \cite{Rie4}, this problem is extremely interesting and important.

Although we have mentioned the spectral truncation problem in the Gromov--Hausdorff framework, another possible approach is to consider the \emph{quantum} Gromov–Hausdorff distance \cite{Rie3,AKK}, as pointed out in \cite{LvS}. We note that these two notions of convergence are not equivalent \cite{KK}. However, in the case we consider in this article, Gromov--Hausdorff convergence still implies quantum Gromov--Hausdorff convergence, as follows from Proposition $2.2.4$ in \cite{KK1}. We note the related notion of propinquity \cite{L1,L2}, primarily formulated for $C^*$-algebras. Examples of convergence of spectral truncations include the classical case of the circle, or more generally, $n$-tori \cite{vS,LvS}. See also \cite{T} for a related work. A variant of the spectral truncation problem, namely the convergence of Peter--Weyl truncations, has been established for compact metric groups \cite{GS} and coamenable compact quantum groups \cite{L}. See also \cite{Rie4} for a related treatment of Fourier truncations.

Arguably, one of the most fundamental and widely studied example of noncommutative geometry is the noncommutative $n$-torus ($n\ge 2$) $\mathcal{A}_{\Theta}$ \cite{Rie0}, where $\Theta$ is a real skew-symmetric $n\times n$ matrix, with the canonical Dirac operator coming from the ergodic action of $\mathbb{T}^n$. It is known from the early work of Rieffel \cite{Rie1} that the canonical spectral triple $\big(\mathcal{A}_{\Theta},L^2(\mathcal{A}_{\Theta},\tau)\otimes\bbc^{2^{\lfloor n/2\rfloor}},D\big)$ induces a quantum metric structure on $\mathcal{A}_{\Theta}$ (see \cite{Li} for more general case of $\theta$-deformations). The spectral truncation problem makes sense in this case, and we establish the Gromov--Hausdoff convergence of the state spaces, when equipped with Connes' distance function, of the operator systems associated to spectral truncations for $\mathcal{A}_{\Theta}$. The convergence also holds with respect to the quantum Gromov--Hausdorff distance as well.

A crucial role in our analysis is played by the action of the torus $\mathbb{T}^n$ on $\mathcal{A}_{\Theta}$ and an appropriate use of Bochner integrals. This enables us to construct an $C^1$-\emph{approximate order isomorphism} introduced in \cite{vS}, which allows us to establish the convergence. A pleasant consequence of our work is that the analysis remains valid in the case of $\Theta=0$, which is the classical situation of the $n$-torus $C^\infty(\mathbb{T}^n)$. Thus, the harmonic analysis machinery employed in \cite{LvS} can be avoided, as it is specific to the commutative setting anyway. It is worth noting, however, that the foundational work \cite{LvS} was the first to move beyond dimension one; our contribution extends this perspective beyond the classical setting to the noncommutative $n$-torus, one of the fundamental examples in noncommutative geometry.

We conclude the article with a further analysis of the structure of the truncated operator systems. We compute the $C^*$-envelope \cite{A1,H} and the propagation number \cite{CS1} of each truncated operator system and show that these agree, up to isomorphism of the $C^*$-envelopes, with the corresponding structures obtained in \cite{LvS} for the classical case of tori. In other words, the structural properties of the truncated operator systems for tori persist under deformation. The structure of duals for the truncated operator systems remains largely unexplored, despite our efforts, much as in the classical case of tori.
\smallskip

The article is organized as follows. In Section $2$, we recall the essential facts on noncommutatve tori, compact quantum metric spaces, and vector measures and Bochner integrals. Section $3$ discusses the candidate for $C^1$-approximate order isomorphism. Section $4$ is the main section, where we prove the convergence. Finally, in Section $5$ we study $C^*$-envelopes and propagation numbers of the truncated operator systems.

\section{Preliminaries}

In this section, we briefly recall the necessary terminology and results that we need in this paper. Appropriate references are provided for the readers interested in delve deeper.

\subsection{Noncommutative tori and Dirac operators}

We recall non-commutative $n$-torus from \cite{Rie0}. Let $\Theta$ be an $n\times n$ real skew-symmetric matrix. The noncommutative $n$-torus $\mathbb{T}_{\Theta}$ is the universal $C^*$-algebra generated by $n$ unitaries $U_1,\ldots,U_n$ satisfying $U_kU_\ell=e^{2\pi i\Theta_{k\ell}}U_\ell U_k$, where $k,\ell\in\{1,\ldots,n\}$. Throughout this paper $i$ will stand for $\sqrt{-1}$. On $\mathbb{T}_{\Theta}$, the compact connected Lie group $\mathbb{T}^n$ acts as follows:
$$(z_1,\ldots,z_n)\,.\,U_j=z_jU_j\,;\quad j=1,\ldots,n.$$
The smooth dense subalgebra of $\mathbb{T}_{\Theta}$ under the above action is
$$\mathcal{A}_{\Theta}:=\big\{\textstyle\sum_{r\in\bbz^n}\,a_rU^r:(a_r)\in\mathbb{S}(\bbz^n),\,r=(r_1,\ldots,r_n)\in\bbz^n\big\},$$
called the smooth noncommutative $n$-torus, where $U^r:=U_1^{r_1}\ldots U_n^{r_n}$, and $\mathbb{S}(\bbz^n)$ denotes vector space of multisequences $(a_r)$ of rapid decay. When $\Theta=0$, we have $\mathbb{T}_0=C(\mathbb{T}^n)$ and $\mathcal{A}_0=C^\infty(\mathbb{T}^n)$.

For $k=(k_1,\ldots,k_n)\in\mathbb Z^n$, we denote $e_k:=U_1^{k_1}\ldots U_n^{k_n}$. We have the following facts:
\begin{lmma}\label{cocycle}
We have the following:
$$e_ke_\ell=\omega(k,\ell)e_{k+\ell}\,,\qquad \omega(k,\ell):=e^{-2\pi i\sum_{1\le r<j\le n}\Theta_{rj} k_j\ell_r}\in \mathbb{S}^1;$$
$$e_k^*=\overline{\omega(k,-k)}\,e_{-k}\,,\qquad e_\lambda e_\nu^*=\overline{\omega(\lambda-\nu,\nu)}\,e_{\lambda-\nu}.$$
Moreover, $$\omega(k,\ell+\ell')=\omega(k,\ell)\omega(k,\ell')\,,\qquad \omega(k+k',\ell)=\omega(k,\ell)\omega(k',\ell)\,,\qquad \omega(k,0)=\omega(0,k)=1,$$
and we have the cocycle identity $\,\omega(k,\ell)\omega(k+\ell,r)=\omega(\ell,r)\omega(k,\ell+r)$.
\end{lmma}

The functional $\tau\big(\sum_ka_ke_k\big):=a_0$ is a faithful trace on $\mathcal{A}_{\Theta}$, invariant under the action, and the family $\{e_k:k\in\mathbb Z^n\}$ is orthonormal for the inner product $\langle a,b\rangle_\tau:=\tau(a^*b)$. The GNS Hilbert space $L^2(\mathcal{A}_{\Theta},\tau)$ is isomorphic to $\ell^2(\mathbb Z^n)$ and we have the faithful representation of $\mathcal{A}_{\Theta}$ on $L^2(\mathcal{A}_{\Theta},\tau)$ by the left-regular representation $\pi_\tau(a)(b):=ab$.

If one further assumes that the lattice $\Lambda_{\Theta}$ generated by columns of $\Theta$ is such that $\Lambda_{\Theta}+\bbz^n$ is dense in $\mathbb{R}^n$, the $C^*$-algebra $\mathbb{T}_{\Theta}$ becomes simple \cite[Page 537]{GVF}. But in our case, simpleness is not needed and we do not require any assumption on $\Theta$ except skew-symmetry.

Let $\{\nu_1,\ldots,\nu_n\}$ be the standard basis of the Lie algebra $\mathbb{R}^n$ of $\mathbb{T}^n$ and the associated derivations be $\{\delta_{\nu_1},\ldots,\delta_{\nu_n}\}$. We denote $\delta_{\nu_j}$ simply by
$\widetilde{\delta}_j$. These derivations are given by
\begin{IEEEeqnarray}{lCl}\label{derivations1}
\widetilde{\delta}_j\big(\textstyle\sum_r a_r e_r\big)=\textstyle\sum_r ir_j a_r e_r.
\end{IEEEeqnarray}
These are $\ast$-derivations and they commute. Denote $N=2^{\lfloor{n/2}\rfloor}$. On the Hilbert space $H:=L^2(\mathcal{A}_{\Theta},\tau)\otimes\bbc^N\cong\ell^2(\bbz^n)\otimes\bbc^N$, consider the unbounded operator $D = \sum_{j=1}^{n}\delta_j \otimes \sigma_j$, where $\delta_j=-i\,\widetilde{\delta}_j$ with $\widetilde{\delta}_j$ as defined in \Eref{derivations1}. Here, $\sigma_1,\ldots,\sigma_n$ are the Clifford gamma matrices in $M_N(\bbc)$ satisfying $\sigma_\ell\sigma_k+\sigma_k\sigma_\ell=2\delta_{\ell k}$ for $\ell,k\in\{1,\ldots,n\}$. Then $D$ is an unbounded essentially self-adjoint operator acting on $H$ with domain $\mathcal{A}_{\Theta}\otimes\bbc^N$, and $D$ has compact resolvent. The tuple $(\mathcal{A}_{\Theta},H,D)$ is an $n^+$-summable spectral triple for $\mathbb{T}_{\Theta}$. For differential calculus for this spectral triple, we refer to \cite{CG}. We have
\begin{IEEEeqnarray}{lCl}\label{commutator}
[D,\pi(a)]=\textstyle\sum_{\mu=1}^n\pi_\tau(\delta_\mu(a))\otimes\sigma_\mu,
\end{IEEEeqnarray}
and
\begin{IEEEeqnarray}{lCl}\label{D2}
D^2=\textstyle\sum_{\mu=1}^n\delta_\mu^2\otimes I_N\,,\qquad D^2(e_k\otimes\xi)=\|k\|^2e_k\otimes\xi,
\end{IEEEeqnarray}
for all $\xi\in\bbc^N$.

For notational brevity, throughout the article, we omit the representation symbol $\pi$ and simply regard $\mathcal{A}_\Theta$ concretely inside $\mathcal{B}\big(L^2(\mathcal{A}_{\Theta},\tau)\otimes\bbc^{2^{\lfloor n/2\rfloor}}\big)$.

\subsection{Spectral metric space and \texorpdfstring{$C^1$}{C1}-approximate order isomorphism}

We briefly recall some preliminaries on compact quantum metric spaces, developed by Rieffel \cite{Rie1,Rie2}, motivated by Connes \cite{C1,C2}. We shall use the operator-system reformulation of Rieffel's construction. For detailed discussion, we refer to \cite{KK1} (or, see \cite{BGB} for a quick overview).

\begin{dfn}
Let $X$ be a unital self-adjoint closed subspace of a unital $C^*$-algebra $\mathscr{A}$, called unital complete operator system. A seminorm $L:X\to[0,\infty]$ is called a Lip-norm if it satisfies the following conditions:
\begin{enumerate}[$(i)$]
\item $L$ is $\ast$-invariant and $L(x)=0$ if and only if $x\in\mathbb{C}$;
\item The set $\mathrm{dom}(L):=\{x\in X:L(x)<\infty\}$ is dense in $X$;
\item The following function (Monge–Kantorovi\v{c} metric)
\[
d^L(\phi,\psi):=\sup\{|\phi(x)-\psi(x)|:L(x)\leq 1\}
\]
metrises the weak $\ast$-topology on the state space $\mathcal{S}(X)$.
\end{enumerate}
The pair $(X,L)$ is called a compact quantum metric space and $L$ is called a Lip-norm.
\end{dfn}

An operator system spectral triple $(E,H,D)$ (see Defn. $1.1$ \cite{LvS}) for which $(\overline{E},L_D)$ is a compact quantum metric space, where $L_D(e):=\norm{[D,e]}$, is called a spectral metric space, or metric spectral triple. Noncommutative tori are spectral metric spaces for the canonical spectral triple discussed in Sec. $2.1$ \cite[Thm. 4.2]{Rie1}.

Let $(X,L_X)$ and $(Y,L_Y)$ be compact quantum metric spaces, and consider the coordinate projections $\,q_X:X\oplus Y\to X$ and $q_Y:X\oplus Y\to Y$. A Lip-norm $L$ on $X\oplus Y$ is called admissible if $\mathrm{dom}(L)=\mathrm{dom}(L_X)\oplus\mathrm{dom}(L_Y)$ and if the quotient semi-norms induced by $L$ on the self-adjoint parts $X_{sa}\mbox{ and }Y_{sa}$ agree with the restrictions of $L_X$ and $L_Y$, respectively. Given such an $L$, the coordinated projections induce embeddings of the state spaces $q_X^*:\mathcal{S}(X)\hookrightarrow \mathcal{S}(X\oplus Y)$ and $q_Y^*:\mathcal{S}(Y)\hookrightarrow \mathcal{S}(X\oplus Y)$. The quantum Gromov--Hausdorff distance between $(X,L_X)$ and $(Y,L_Y)$ is defined by
$$
\mathrm{dist}_q\big((X,L_X),(Y,L_Y)\big)
:=\inf\left\{\mathrm{dist}_H^L\big(q_A^*(\mathcal{S}(X)),q_B^*(\mathcal{S}(Y))\big):L\,\,\mathrm{admissible}\right\}
$$
where $\mathrm{dist}_H^L$ is the Hausdorff distance measured with respect to $d^L$. This definition agrees with Riffel's original definition in \cite{Rie3} applied to the order unit spaces $\mathrm{dom}(L_X)_{sa}$ and $\mathrm{dom}(L_Y)_{sa}$, see \cite[Lemma 2.2.2]{KK1}.

It follows that $\mathrm{dist}_q\big((X,L_X),(Y,L_Y)\big)\geq \mathrm{dist}_{GH}\big((\mathcal{S}(X),d^{L_X}),(\mathcal{S}(Y),d^{L_Y})\big)$, where $\mathrm{dist}_{GH}$ denotes the classical Gromov--Hausdorff distance. However, these two metrics (formally speaking, on isometric classes) are not equivalent \cite{KK}.

\begin{dfn}[\cite{vS}]\label{dfn:orderiso}
Let $(E,H,D)$ and $(E_{\Lambda},H_{\Lambda},D_{\Lambda})$, for all $\Lambda\ge 0$, be operator system spectral triples such that the topologies on the state spaces $\mathcal{S}(E)$ and $\mathcal{S}(E_{\Lambda})$, for each $\Lambda\ge 0$, defined by Connes' distance function agree with the weak $\ast$-topologies on them. Linear maps $\rho_{\Lambda}:E\to E_{\Lambda}$ and $\sigma_{\Lambda}:E_{\Lambda}\to E$ form a \emph{$C^1$-approximate order isomorphism} if
\begin{enumerate}[$(i)$]
\item $\rho_{\Lambda}\mbox{ and }\sigma_{\Lambda}$ are unital and positive;
\item $\rho_{\Lambda}\mbox{ and }\sigma_{\Lambda}$ are $C^1$-contractive, i.e. they are contractive both with respect to norm and with respect to Lipschitz seminorm $\|[D,\,\cdot\,]\|$, respectively $\|[D_{\Lambda},\,\cdot\,]\|$, i.e.,
$$\|[D_{\Lambda},\rho_{\Lambda}(e)]\|\le\|[D,e]\|\quad(e\in E),$$ $$
\|[D,\sigma_{\Lambda}(T)]\|\le\|[D_{\Lambda},T]\|\quad(T\in E_{\Lambda}).$$
\item we have $\|e-\sigma_{\Lambda}\circ\rho_{\Lambda}(e)\|\le\gamma_{\Lambda}\|[D,e]\|$ and $\|T-\rho_{\Lambda}\circ\sigma_{\Lambda}(T)\|\le\gamma_{\Lambda}\|[D_{\Lambda},T]\|$ for some net $(\gamma_{\Lambda})_{\Lambda\ge 0}$ with $\gamma_{\Lambda}\to 0$ as $\Lambda\to\infty$.
\end{enumerate}
We call the collection of pairs of maps $(\rho_{\Lambda},\sigma_{\Lambda})$ a $C^1$-approximate order isomorphism.
\end{dfn}

The following result is an excellent device for proving convergence of spectral truncations in the Gromov--Hausdorff distance. The quantum Gromov--Hausdorff convergence then automatically follows from Proposition $2.2.4$ in \cite{KK1}.

\begin{thm}[Thm. 5, \cite{vS}]\label{Thmofsuijlekom}
Suppose $(E,H,D)$ and $\{(E_R,H_R,D_R)\}_{R\geq 0}$ are operator system spectral triples such that the topologies on the state spaces $\mathcal{S}(E_R)$ and $\mathcal{S}(E)$ defined by the Connes' distance functions $d_{E_R}$ and $d_E$, respectively, agree with the weak $\ast$-topologies on them. If $(\rho_R,\sigma_R)$ is a $C^1$-approximate order isomorphism for $(E,H,D)$ and $\{(E_R,H_R,D_R)\}_{R\geq 0}$, then the state spaces $\{(\mathcal{S}(E_R),d_{E_R})\}_{R\geq 0}$ converge to $(\mathcal{S}(E),d_{E})$ in the Gromov--Hausdorff distance.
\end{thm}

\subsection{Vector valued measure and Bochner integral}

We recall some necessary results on vector valued measures and Bochner integrals that are needed in this article. For general facts, we refer to \cite{D,DU}.

\begin{lmma}[8.6.4, \cite{D}]\label{P3}
Let $X$ be a Banach space and $f_k:[0,1]\to X$ continuously differentiable functions such that (i) $\sum_k f_k(t_0)$ converges in $X$ for some $t_0\in[0,1]$, and (ii) $\sum_k f_k'(t)$ converges uniformly on $[0,1]$. Then $\sum_k f_k(t)$ converges uniformly on $[0,1]$ to a $C^1$ function $g$, and $g'(t)=\sum_k f_k'(t)$ for all $t$.
\end{lmma}

\begin{thm}[Theorem 2, Section 2.1, \cite{DU}]\label{P4'}
Let $(\Omega,\Sigma,\mu)$ be a finite measure space and let
$X$ be a complex Banach space. A function $f:\Omega\to X$ is $\mu$-measurable if and only if
\begin{enumerate}[$(i)$]
\item $f$ is $\mu$-essentially separably valued, i.e., there exists $E\in\Sigma$ with $\mu(E)=0$ and such that $f(\Omega\setminus E)$ is a (norm) separable subset of $X$, and
\item $f$ is weakly $\mu$-measurable.
\end{enumerate}
\end{thm}

\begin{crlre}\label{P4''}
Suppose that $\Omega$ is a compact metrizable space and that $\Sigma$
contains the Borel $\sigma$-algebra of $\Omega$. If
$f:\Omega\to X$ is continuous, then $f$ is $\mu$-measurable.
\end{crlre}
\begin{proof}
For every $x^*\in X^*$, the function
$x^*\circ f:\Omega\to\mathbb{C}$ is continuous and hence
$\mu$-measurable. Thus $f$ is weakly $\mu$-measurable. Since $\Omega$ is
compact metrizable, it is separable, and therefore $f(\Omega)$ is
separable. Hence $f$ is $\mu$-essentially separably valued by taking
$E=\varnothing$. The result follows from Theorem \ref{P4'}.
\end{proof}

\begin{thm}[Theorem 2, Section 2.2, \cite{DU}]\label{P4'''}
A $\mu$-measurable function $f:\Omega\to X$ is Bochner integrable if and only if $\int_{\Omega} \|f\|  \ d\mu< \infty$.
\end{thm}

\begin{lmma}\label{P4}
If $X$ is a Banach space and $f:\Omega\to X$ is continuous on a compact metric space $\Omega$ equipped with a finite Borel measure $\mu$, then $f$ is Bochner integrable.
\end{lmma}
\begin{proof}
By Corollary \ref{P4''}, $f$ is $\mu$-measurable; since $\Omega$ is compact and $f$ continuous, $\|f\|$ is bounded, so $\int_\Omega\|f\|\,d\mu\le \mu(\Omega)\sup_\Omega\|f\|<\infty$. Apply Theorem \ref{P4'''}.
\end{proof}

\begin{thm}[Theorem 9, Section 2.2,\cite{DU}]\label{P5'}
Let $f$ be Bochner integrable on $[0, 1]$ with respect to Lebesgue measure. Then for almost all $s \in [0, 1]$ one has $$ \lim_{h \to 0}\,\frac{1}{h} \int_{s}^{s+h} \|f(t) - f(s)\| \, dt = 0.$$ Consequently, for almost all $s \in [0, 1]$ one has $$ \lim_{h \to 0}\,\frac{1}{h} \int_{s}^{s+h} f(t) \, dt = f(s).$$
\end{thm}

For the following lemma, for each $t\in[0,1]$, the derivative of $g:\mathbb{R}\to X$ is a bounded linear map $ g'(t)\in\mathcal B(\mathbb R,X).$ There is a canonical isometric isomorphism $J:X\to\mathcal B(\mathbb R,X)$, defined by $J(x)(s):=sx$, with inverse $J^{-1}(f):=f(1)$. Under this identification, we regard $g'(t)$ as the vector $g'(t)(1)\in X$. Consequently, the norm-continuity of $t\mapsto g'(t)\in \mathcal B(\mathbb R, X)$ is equivalent to the norm-continuity of $t\mapsto g'(t)(1)\in X$. Moreover, as $\|g'(t)\|=\|g'(t)(1)\|$, for all $t\in \mathbb R$, so $g'$ Bochner integrable if and only if $g'(\cdot)(1)$ is Bochner integrable.

\begin{lmma}\label{P5}
If $g:\mathbb{R}\to X$ such that $g'$ norm-continuous, then $g'$ is Bochner integrable over $[0,1]$ and $g(1)-g(0)=\int_0^1 g'(t)\,dt$.
\end{lmma}
\begin{proof}
Define $\phi(t):=\int_{0}^t g'(s)\,ds$. By Theorem \ref{P5'}, $\phi$ is differentiable and $\phi'(t)=g'(t)$. Hence $(g-\phi)'=0$, so $g-\phi$ is constant. Therefore $g(1)-g(0)=\phi(1)-\phi(0)=\int_0^1 g'(t)\,dt$.
\end{proof}

\begin{rmrk}\rm
If we remove the identification stated above, then if $g:\mathbb{R}\to X$ with $t\to g'(t)\in \mathcal{B}(\mathbb{R},X)$ norm-continuous, then over $[0,1]$, $t\to g'(t)(1)\in X$ is Bochner integrable and $g(1)-g(0)=\int_0^1 g'(t)(1)\,dt$. Still it is provable by taking $\phi(t):=\int_0^t g'(s)(1)\,ds$, then again $\phi'(t)=g'(t)\in \mathcal B(\mathbb R,X)$, because $\|\phi(t+h)-\phi(t)-g'(t)(h)\|/|h| \le \sup_{|s-t|\le |h|}\|g'(s)(1)-g'(t)(1)\|.$ For proper justification, see \cite[Page 335]{D}.
\end{rmrk}

\begin{lmma}[Theorem 4, Section 2.2, \cite{DU}]\label{P6}
For Bochner-integrable $f:\Omega\to X$, $\big\|\int_\Omega f\,d\mu\big\|\le\int_\Omega\|f\|\,d\mu$.
\end{lmma}

\begin{thm}[Theorem 3, Section 2.2, \cite{DU}]\label{P7'}
Let $(\Omega,\Sigma,\mu)$ be a finite measure space and let $(f_n)$ be a sequence of Bochner integrable $X$-valued functions on $\Omega$. If $\lim_{n\to\infty} f_n=f \  \text{in $\mu$-measure},$ i.e., $\lim_{n\to\infty}
\mu\bigl(\{\omega\in\Omega:\norm{f_n-f}>\varepsilon\}\bigr)=0
\ \text{for every }\varepsilon>0,$ and if there exists a real-valued Lebesgue ($\mu$-) integrable function $g$ on $\Omega$ with $\norm{f_n}\leq g \ \mu\text{-almost everywhere},$ then $f$ is Bochner integrable and $\lim_{n\to\infty}\int_E f_n\,d\mu=\int_E f\,d\mu \text{ for each }E\in\Sigma.$ In fact, $\lim_{n\to\infty}\int_\Omega \norm{f-f_n}\,d\mu=0$.
\end{thm}

\begin{crlre}\label{P7}
In the same setup of Theorem \ref{P7'}, if $f_n\to f$ uniformly on $\Omega$, with each $f_n$ Bochner integrable, then $f$ is Bochner integrable and $\int_\Omega f_n\,d\mu\to\int_\Omega f\,d\mu$.
\end{crlre}
\begin{proof}
As $f_n$ converges to $f$ uniformly, convergence in $\mu$-measure follows. Also, there exists $N\in\mathbb{N}$ such that $\|f_n(\omega)-f_N(\omega)\|\leq 1$ for all $n\geq N$ and $\omega\in\Omega$. Let $g:\Omega\to\mathbb{R}$ be defined by $g(\omega):=1+\sum_{1\leq k\leq N}\|f_k(\omega)\|.$ Since each $f_k$ is Bochner integrable, $\|f_k\|$ is $\mu$ integrable. Hence, $g$ is $\mu$ integrable. Moreover $\|f_n\|\leq g$ for all $n\in\mathbb{N}$.
\end{proof}

\begin{thm}[Theorem 6, Section 2.2, \cite{DU}]\label{P7''}
Let $T$ be a closed linear operator defined inside $X$ and having values in a Banach space $Y$. If $f$ and $Tf$ are Bochner integrable with respect to $\mu$, then
$$T\left(\int_E f\,d\mu\right)=\int_E Tf\,d\mu$$
for all $E\in\Sigma$.
\end{thm}

\begin{lmma}\label{P8}
If $f_k:\Omega\to X$ with $\|f_k(x)\|\le M_k$ for all $x\in\Omega$ and $\sum_k M_k<\infty$, then $\sum_k f_k$ converges uniformly on $\Omega$.
\end{lmma}
\begin{proof}
Verbatim from the real or complex Weierstrass $M$-test, using completeness of $X$  and the Cauchy criterion for uniform convergence.
\end{proof}

\section{A candidate for \texorpdfstring{$C^1$}{C1}-approximate order isomorphism}
\subsection{The maps \texorpdfstring{$\sigma_R,\,\rho_R$}{sigmaRrhoR} and quantum metric structure on truncated operator systems}

Consider the canonical spectral triple $(\mathcal{A}_{\Theta},L^2(\mathcal{A}_\Theta,\tau)\otimes\bbc^{2^{\lfloor{n/2}\rfloor}},D)$ for the noncommutative torus described in Section $2.1$. For $R\ge 0$ put $\Lambda_R:=\{k\in\mathbb Z^n:\|k\|\le R\}$, and $N_R:=\#\Lambda_R$. The spectral projection of $D$ onto $[-R,R]$ is
\begin{IEEEeqnarray}{lCl}\label{abc}
P_R=E_{D}([-R,R])=E_{D^2}([0,R^2])=E_{\sum_{j=1}^n\delta_j^2}([0,R^2])\otimes I_{2^{\lfloor{n/2}\rfloor}}=P_R^0\otimes I_{2^{\lfloor{n/2}\rfloor}}
\end{IEEEeqnarray}
where $P_R^0$ is the orthogonal projection onto $\mathrm{span}\{e_k:k\in\Lambda_R\}\subseteq L^2(\mathcal{A}_\Theta,\tau)$. We have rank$(P_R^0)=N_R$ and $H_R:=P_R(L^2(\mathcal{A}_\Theta,\tau)\otimes \mathbb{C}^{2^{\lfloor{n/2}\rfloor}})\cong P_R^0\,L^2(\mathcal{A}_\Theta,\tau)\otimes \mathbb{C}^{2^{\lfloor{n/2}\rfloor}}=:H_R^0\otimes\bbc^{2^{\lfloor{n/2}\rfloor}}$. Consequently, $\dim(H_R)=\text{rank}(P_R)=2^{\lfloor{n/2}\rfloor}N_R$ and $\dim(H_R^0)=N_R$. Obviously, each $P_R$ commutes with $D$.

Denote $E_R:= P_R^0 \pi_\tau (\mathcal{A}_\Theta)P_{R}^0\subseteq \mathcal B(H_R^0)\cong M_{N_R}(\mathbb{C})$. For $T\in\mathcal{B}(H_R^0)$, define $T_{\lambda\nu}:=\langle e_\lambda,Te_\nu\rangle$, $\lambda,\nu\in\Lambda_R$. We write $D_R:=P_RDP_R$, and consequently for the operator system spectral triple $(E_R, H_R, D_R)$, we have $\pi_R:E_R\to \mathcal B(H_R)$ defined as $\pi_R(T):= T\otimes I_{2^{\lfloor{n/2}\rfloor}}$. We reserve these notations throughout the article:
\begin{IEEEeqnarray}{lCl}\label{notn}
\hspace*{-1cm}E_R:=P_R^0\pi_\tau(\mathcal{A}_{\Theta})P_R^0\,,\,\,H_R:=P_R^0\,L^2(\mathcal{A}_\Theta,\tau)\otimes \mathbb{C}^{2^{\lfloor{n/2}\rfloor}}=H_R^0\otimes\bbc^{2^{\lfloor{n/2}\rfloor}}\,,\,\,D_R:=P_RDP_R
\end{IEEEeqnarray}
where $P_R=P_R^0\otimes I_{2^{\lfloor{n/2}\rfloor}}$ as in \Eref{abc}.

Motivated by the classical case of tori \cite{LvS}, we define the maps $\rho_R:\mathcal{A}_{\Theta}\to E_R\subseteq \mathcal B(H_R^0)$ and $\sigma_R: \mathcal B(H_R^0) \to \mathcal{A}_{\Theta}$ as
\begin{IEEEeqnarray}{lCl}\label{eq:1}
\rho_R(a):=P_R^0\pi_\tau(a)P_R^0\,,\qquad \sigma_R(T):=\frac1{N_R}\sum_{\lambda,\nu\in\Lambda_R}T_{\lambda\nu}\,e_\lambda e_\nu^*.
\end{IEEEeqnarray}
Observe that $\pi_R\circ\rho_R(a)=P_RaP_R$ for all $a\in\mathcal{A}_{\Theta}$.

\begin{lmma}\label{matrixcoeff}
For $a=\sum_{k}a_ke_k\in\mathcal{A}_{\Theta}$, if $\rho_R(a):=T$, then
\[
T_{\lambda\nu}=a_{\lambda-\nu}\,\omega(\lambda-\nu,\nu),\quad\lambda,\nu\in\Lambda_R.
\]
\end{lmma}
\begin{proof}
We have $$T_{\lambda\nu}=\langle e_\lambda,\pi_\tau(a)e_\nu\rangle=\tau(e_\lambda^*ae_\nu)=\textstyle\sum_ka_k\tau(e_\lambda^*e_ke_\nu)=\textstyle\sum_ka_k\omega(k,\nu)\tau(e_\lambda^*e_{k+\nu}),$$ which shows that $\tau(e_\lambda^*e_{k+\nu})=\delta_{\lambda,k+\nu}$, forcing $k=\lambda-\nu$.
\end{proof}

Now, we present a general result. For Hilbert spaces $H,K$ and a unit vector $\xi\in K$, the map $V_\xi:H\to H\otimes K$, given by $V_\xi(\eta):=\eta\otimes\xi$, is a linear isometry. For $T\in \mathcal B(H)\mbox{ and }S\in \mathcal B(K)$,
$$V_\xi^*(T\otimes S)V_\xi=\langle S\xi\,,\,\xi\rangle\,T.$$

\begin{lmma}\label{spinorcompression}
Let $n\geq 2$ and $K$ be a Hilbert space. Consider
$X=\sum_{j=1}^n\,X_j\otimes\sigma_j\in \mathcal B\big(K\otimes\mathbb C^{2^{\lfloor n/2\rfloor}}\big)$
for some $X_1,\ldots,X_n\in \mathcal B(K)$, and
$\sigma_1,\ldots,\sigma_n\in
\mathcal B\left(\mathbb C^{2^{\lfloor n/2\rfloor}}\right)$ are self-adjoint matrices satisfying
$\sigma_\mu^2=I$ and
$\sigma_\mu\sigma_\nu+\sigma_\nu\sigma_\mu=0$
for $\mu\ne\nu$.
For $\mu=1,\ldots,n$, let
$\{\xi_\mu^{j,+}:1\le j\le 2^{\lfloor n/2\rfloor-1}\}$ and
$\{\xi_\mu^{j,-}:1\le j\le 2^{\lfloor n/2\rfloor-1}\}$
be orthonormal bases of the $+1$ and $-1$ eigenspaces of $\sigma_\mu$, respectively. Let
$V_\mu^{j,\pm}:K\to K\otimes\mathbb C^{2^{\lfloor n/2\rfloor}}$
be the isometry $\eta\mapsto\eta\otimes\xi_\mu^{j,\pm}$, and define
$$
\Phi_\mu(T):=
\frac{1}{2^{\lfloor n/2\rfloor}}
\sum_{j=1}^{2^{\lfloor n/2\rfloor-1}}
\left(
(V_\mu^{j,+})^*TV_\mu^{j,+}
-
(V_\mu^{j,-})^*TV_\mu^{j,-}
\right).
$$
Then
$$
\Phi_\mu(X)=X_\mu,\text{ and so }\|X_\mu\|\le\|X\|,
\qquad \mu=1,\ldots,n.
$$
\end{lmma}
\begin{proof}
Fix any $\mu\in\{1,\ldots,n\}$. Since $n\ge2$, choose $\nu\ne\mu$. The anticommutation relation implies that $\sigma_\nu$ maps the $+1$ eigenspace of $\sigma_\mu$ unitarily onto the $-1$ eigenspace of $\sigma_\mu$. Hence these eigenspaces have the same dimension, namely $2^{\lfloor n/2\rfloor-1}$.

Since
$\sigma_\mu\xi_\mu^{j,\pm}=\pm\xi_\mu^{j,\pm}$, we have
$$
(V_\mu^{j,\pm})^*XV_\mu^{j,\pm}
=
\sum_{\nu=1}^n
\langle\sigma_\nu\xi_\mu^{j,\pm},\xi_\mu^{j,\pm}\rangle X_\nu.
$$
For $\nu\ne\mu$, using the anticommutation relation, we obtain
$\sigma_\mu\sigma_\nu=-\sigma_\nu\sigma_\mu$ and therefore
$$
\sigma_\mu(\sigma_\nu\xi_\mu^{j,\pm})
=
-\sigma_\nu\sigma_\mu\xi_\mu^{j,\pm}
=
\mp\sigma_\nu\xi_\mu^{j,\pm}.
$$
Thus $\sigma_\nu\xi_\mu^{j,\pm}$ belongs to the $\mp1$ eigenspace of $\sigma_\mu$, whereas $\xi_\mu^{j,\pm}$ belongs to the $\pm1$ eigenspace. Since $\sigma_\mu$ is self-adjoint, these eigenspaces are orthogonal. Hence
$$
\langle\sigma_\nu\xi_\mu^{j,\pm},\xi_\mu^{j,\pm}\rangle=0,
\qquad \nu\ne\mu.
$$
On the other hand,
$\langle\sigma_\mu\xi_\mu^{j,+},\xi_\mu^{j,+}\rangle=1$
and
$\langle\sigma_\mu\xi_\mu^{j,-},\xi_\mu^{j,-}\rangle=-1$.
Therefore,
$(V_\mu^{j,+})^*XV_\mu^{j,+}=X_\mu$
and
$(V_\mu^{j,-})^*XV_\mu^{j,-}=-X_\mu$.
Consequently,
$$
\begin{aligned}
\Phi_\mu(X)
&=
\frac{1}{2^{\lfloor n/2\rfloor}}
\sum_{j=1}^{2^{\lfloor n/2\rfloor-1}}
\bigl(X_\mu-(-X_\mu)\bigr)\\
&=
\frac{1}{2^{\lfloor n/2\rfloor}}
\,2^{\lfloor n/2\rfloor-1}\,2X_\mu
=X_\mu.
\end{aligned}
$$

Finally, since $V_\mu^{j,\pm}$ is an isometry,
$\|(V_\mu^{j,\pm})^*TV_\mu^{j,\pm}\|\le\|T\|$. Hence
$$
\begin{aligned}
\|\Phi_\mu(T)\|
&\le
\frac{1}{2^{\lfloor n/2\rfloor}}
\sum_{j=1}^{2^{\lfloor n/2\rfloor-1}}
\left(
\|(V_\mu^{j,+})^*TV_\mu^{j,+}\|
+
\|(V_\mu^{j,-})^*TV_\mu^{j,-}\|
\right)\\
&\le
\frac{1}{2^{\lfloor n/2\rfloor}}
\,2^{\lfloor n/2\rfloor-1}\,2\|T\|
=\|T\|.
\end{aligned}
$$
Taking $T=X$ gives
$\|X_\mu\|=\|\Phi_\mu(X)\|\le\|X\|$.
\end{proof}

On $H_R^0=\mathrm{span}\{e_\lambda:\lambda=(\la_1,\ldots,\la_n)\in\Lambda_R\}$, define
\begin{IEEEeqnarray}{lCl}\label{to refer in 4.2}
L_\mu e_\lambda:=\lambda_\mu e_\lambda,\qquad\text{i.e.}\qquad L_\mu:=P_R^0\,\delta_\mu\,P_R^0,
\end{IEEEeqnarray}
for $\mu=1,\ldots,n$. Thus, $L_\mu$ is the restriction of the derivation $\delta_\mu$ to the truncated space. Also observe that $L_{\mu}$ commutes with each other, being diagonal matrices.

\begin{lmma}\label{deltabound1}
We have $D_R=\sum_{j=1}^n\,L_j\otimes\sigma_j$, and consequently, for $T\in E_R\subseteq\mathcal B(H_R^0)$, $$[D_R,T\otimes I_{2^{\lfloor{n/2}\rfloor}}]=\sum_{j=1}^n[L_j,T]\otimes\sigma_j,$$ with $[L_\mu,T]_{\lambda\nu}=(\lambda_\mu-\nu_\mu)T_{\lambda\nu}$. Furthermore, we have $\|[L_\mu,T]\|\le\|[D_R,T\otimes I_{2^{\lfloor{n/2}\rfloor}}]\|$.
\end{lmma}
\begin{proof}
Since $P_R=P_R^0\otimes I_{2^{\lfloor{n/2}\rfloor}}$, we have $D_R=\sum_{j=1}^nL_j\otimes\sigma_j$.
Since $L_\mu$ is diagonal with $L_\mu e_\lambda=\lambda_\mu e_\lambda$, so if $T=[T_{\la\nu}]_{\la,\nu\in\Lambda_R}$, then $(L_\mu T)_{\lambda\nu}=\lambda_\mu T_{\la\nu}$ and $(TL_\mu)_{\lambda\nu}=T_{\la\nu}\nu_\mu$. Hence, $[L_\mu,T]_{\lambda\nu}=(\lambda_\mu-\nu_\mu)T_{\la\nu}$, and the desired expression of the commutator follows. Finally, applying Lemma \ref{spinorcompression} with $K=H_R^0$, the required inequality follows.
\end{proof}

\begin{ppsn}
$(E_R,L_R)$ are compact quantum metric spaces for all $R\ge 0$, where $L_R(T):=\|[D_R,T\otimes I_{2^{\lfloor{n/2}\rfloor}}]\|$, for $T\in E_R$.
\end{ppsn}
\begin{proof}
Clearly $L_R(T)=L_R(T^*)$, and we show that $\ker(L_R)=\mathbb{C}.\, I_{N_R}$. Suppose $[D_R, T\otimes I_{2^{\lfloor{n/2}\rfloor}}]=0$. By \Lref{deltabound1}, we get $[L_\mu,T]=0$ for $\mu=1,\ldots,n$, which further implies that $[L_\mu,T]_{\lambda\nu}=(\lambda_\mu-\nu_\mu)T_{\lambda\nu}=0$ for all $\la, \nu \in \Lambda_R$.  So when $\la\neq \nu$, we have $T_{\la\nu}=0$, and if $T=\rho_R(a)\in E_R$, then for all $\la\in\Lambda_R$, we have $T_{\la\la}=a_{\la-\la}w(\la-\la,\la)=a_0$ by \Lref{matrixcoeff}. This implies that $T=a_0 I_{N_R}$, and consequently $T$ is scalar. The converse direction is obvious. Therefore, $L_R$ induces a norm on $E_R/\mathbb{C}.\, I_{N_R}$, and because of finite-dimensionality it is equivalent to the quotient norm $\|.\|_q$. By \cite[Theorem 1.8]{Rie1}, the metric induced by \(L_R\) on the state space \(\mathcal{S}(E_R)\) metrizes the weak $\ast$-topology on it.
\end{proof}

\subsection{Compositions and \texorpdfstring{$C^1$}{C1}-contractivity of
\texorpdfstring{$\rho_R$ and $\sigma_R$}{(rhoR, sigmaR)}}

In this subsection, we obtain the precise expressions for the compositions $\sigma_R\circ\rho_R$ and $\rho_R\circ\sigma_R$. This is followed by proving the $C^1$-contractivity of the maps $\rho_R$ and $\sigma_R$.

Recall the following result from \cite{LvS}.

\begin{lmma}\label{mR(k)}
For $k\in\mathbb Z^n$, let $N_L(\Lambda_R,k):=\#\{(\lambda,\nu)\in\Lambda_R^2:\lambda-\nu=k\}$ and $m_R(k):=N_L(\Lambda_R,k)/N_R$, then
\begin{enumerate}[$(i)$]
\item $m_R(k)\in[0,1]$.
\item $m_{R}(k)=0$ if $k\notin \Lambda_{R}-\Lambda_{R}$.
\item $m_R(k)=m_{R}(-k)$.
\item Moreover define $L(k):=\overline{B}(0,R)\cap \overline{B}(k,R)\cap \mathbb{Z}^n$, and $N_L'(R,k)=\#L(k)$. Then $m_R(k)=N_L'(R,k)/N_R$.
\item For fixed $k\in\mathbb Z^n$, $m_R(k)\to1$ as $R\to\infty$.
\end{enumerate}
\end{lmma}
\begin{proof}
For $(i)$, obviously for each $\nu\in \Lambda_R$, there is atmost one $\la\in \Lambda_R$ such that $\la-\nu=k$. Parts $(ii)$ and $(iii)$ are easy. For $(iv)$, let $A(k):=\{(\lambda,\nu)\in\Lambda_R^2:\lambda-\nu=k\}$. Then $\phi:A(k)\to L(k)$ defined by $\phi((\la,\nu)):=\la$, is a well-defined bijection with inverse $\la\mapsto (\la,\la-k).$ Part $(v)$ follows from \cite[Lemma 3.7]{LvS}.
\end{proof}

\begin{ppsn}\label{compositions}
We have
\begin{IEEEeqnarray}{lCl}\label{eq:3.3}
\sigma_R\circ\rho_R=F_{m_R},\qquad\text{where }\, F_{m_R}\big(\textstyle\sum_ka_ke_k\big):=\textstyle\sum_km_R(k)a_ke_k,
\end{IEEEeqnarray} and
\begin{IEEEeqnarray}{lCl}\label{eq:3.4}
\rho_R\circ\sigma_R=S_{m_R}\ \text{on }E_R,\qquad\text{where }\,\big(S_{m_R}(T)\big)_{\lambda\nu}:=m_R(\lambda-\nu)T_{\lambda\nu}.
\end{IEEEeqnarray}
\end{ppsn}
\begin{proof}
Let $a=\sum_ka_ke_k$. Then
$$\sigma_R\rho_R(a)=\frac{1}{N_R}\sum_{\la,\nu\in \Lambda_R}(\rho_R(a))_{\la\nu}e_\la e_\nu^*.$$
By Lemma \ref{matrixcoeff} and  \ref{cocycle}, we get
$(\rho_{R}(a))_{\la\nu}=a_{\lambda-\nu}\,\omega(\lambda-\nu,\nu)$ and $e_\lambda e_\nu^*=\overline{\omega(\lambda-\nu,\nu)}\,e_{\lambda-\nu}.$ 
We have
$$\sigma_R\rho_R(a)=\frac1{N_R}\sum_{\lambda,\nu\in\Lambda_R}{a_{\lambda-\nu}\omega(\lambda-\nu,\nu)}\cdot{\overline{\omega(\lambda-\nu,\nu)}}_{}\,e_{\lambda-\nu}=\frac1{N_R}\sum_{\lambda,\nu\in\Lambda_R}a_{\lambda-\nu}e_{\lambda-\nu}.$$
The phase cancels exactly, for any $\Theta$. Grouping by $k=\lambda-\nu$ (there are exactly $N_L(\Lambda_R,k)$ such pairs) gives $\sigma_R\rho_R(a)=\sum_km_R(k)a_ke_k=F_{m_R}(a)$, and thus proving \eqref{eq:3.3}.

For \eqref{eq:3.4}, let $T=[T_{\la\nu}]_{\la,\nu}\in E_R\subseteq\mathcal B(H_R^0)$, $T=\rho_R(a)$, so that $T_{\la\nu}=a_{\lambda-\nu}\omega(\lambda-\nu,\nu)$. By \eqref{eq:1}, $\sigma_R(T)=F_{m_R}(a)$, with Fourier coefficients $(\sigma_R(T))_k=m_R(k)a_k$. So, by \Lref{matrixcoeff} we have $$(\rho_R(\sigma_R(T)))_{\lambda\nu}=(\sigma_R(T))_{\lambda-\nu}\,\omega(\lambda-\nu,\nu)=m_R(\lambda-\nu)a_{\lambda-\nu}\,\omega(\lambda-\nu,\nu)=m_R(\lambda-\nu)T_{\lambda\nu}.$$ Hence $\rho_R\circ\sigma_R(T)=S_{m_R}(T)$.
\end{proof}

Now we turn to prove the $C^1$-contractivity of $(\rho_R,\, \sigma_R)$. Recall that for any projection $P\in\mathcal{B}(H)$ and $\pi$ a unital $\ast$-representation of $A$ on $H$,
the map $a\mapsto P\pi(a)P$ is completely positive and contractive, i.e.
\begin{IEEEeqnarray}{lCl}\label{P13}
\|P\pi(a)P\|\le\|\pi(a)\|,\mbox{ or, more generally, }\|PXP\|\le\|X\|,
\end{IEEEeqnarray}
for every $X\in \mathcal{B}(H)$.

\begin{ppsn}\label{rho}
For each $R\ge 0$, the map $\rho_R$ is a unital, completely positive, $C^1$-contractive map.
\end{ppsn}
\begin{proof}
Unitality, complete positivity, and norm-contractivity ($\|\rho_R(a)\|\le\|\pi_{\tau}(a)\|=\|a\|$) follow from \Eref{P13}, since $\rho_R$ is a compression by the projection $P_R^{0}$ of the $*$-representation $\pi_{\tau}$. Since $P_R$ commutes with $D$, we get from \Eref{notn} the following:
$$[D_R\,,\,\pi_R\rho_R(a)]=D_RP_RaP_R-P_RaP_RD_R=P_R[D,a]P_R,$$
and so by \Eref{P13}, we obtain $\|[D_R,\pi_R\rho_R(a)]\|\le\|[D,a]\|$.
\end{proof}

To prove the same for the map $\sigma_R$, recall the following result.
\begin{thm}[Theorem 3.14, \cite{P}]\label{P14}
Let $\mathcal{B}$ be a $C^*$-algebra and $\phi:M_n(\mathbb{C})\to\mathcal{B}$ be a linear map. Suppose
$\{E_{i,j}\}_{i,j=1}^n$ denote the standard matrix units for $M_n(\mathbb{C})$.
The following are equivalent:
\begin{enumerate}[$(i)$]
\item $\phi$ is completely positive.
\item $\phi$ is $n$-positive.
\item $\bigl(\phi(E_{i,j})\bigr)_{i,j=1}^n$ is positive in $M_n(\mathcal{B}).$
\end{enumerate}
\end{thm}

\begin{lmma}\label{sigmaCP}
For each $R\ge 0$, the map $\sigma_R:\mathcal B(H_R^0)\to \mathcal{A}_{\Theta}$ is unital and completely positive. Consequently, $\sigma_R$ is completely contractive.
\end{lmma}
\begin{proof}
To see $\sigma_R$ is unital, observe that 
$\sigma_R(I)=\frac{1}{N_R}\sum_\lambda e_\lambda e_\lambda^*
=\frac{1}{N_R}\sum_\lambda1=1$, since each $e_\lambda$ is unitary.

To prove complete positivity, we use $(iii)$ of Theorem \ref{P14}. Fix an enumeration $\Lambda_R=\{\lambda_1,\ldots,\lambda_{N_R}\}$ and let
$E_{ij}\in M_{N_R}(\mathbb C)$ be the standard matrix units. Since $M_{N_R}(\bbc)$ is $\ast$-isomorphic to $\mathcal{B}(H_R^0)$ via $E_{ij}\mapsto \widetilde{E}_{\lambda_i\lambda_j}$, where $\widetilde{E}_{\lambda_i\lambda_j}(e_{\lambda_k}):=\delta_{jk}e_{\lambda_i}$ with $\delta_{jk}$ being the Kronecker's delta function. Then $ \sigma_R(E_{ij})=\frac{1}{N_R}e_{\lambda_i}e_{\lambda_j}^*.$ Hence, $$\bigl(\sigma_R(E_{ij})\bigr)_{i,j=1}^{N_R} = \frac{1}{N_R}\bigl[e_{\lambda_i}e_{\lambda_j}^*\bigr]_{i,j=1}^{N_R}. $$
Let $$ Y= \begin{pmatrix} e_{\lambda_1} & 0 & \cdots & 0\\ e_{\lambda_2} & 0 & \cdots & 0\\ \vdots & \vdots & \ddots & \vdots\\ e_{\lambda_{N_R}} & 0 & \cdots & 0
\end{pmatrix} \in M_{N_R}(\mathcal{A}_{\Theta}). $$ Then $YY^*=\bigl[e_{\lambda_i}e_{\lambda_j}^*\bigr]_{i,j=1}^{N_R},$ and so $ \bigl(\sigma_R(E_{ij})\bigr)_{i,j=1}^{N_R} =\frac{1}{N_R}YY^*\ge0.$ Thus, by Theorem \ref{P14}, $\sigma_R$ is completely positive.
\end{proof}

\begin{rmrk}\rm
In \cite{LvS}, completely positivity of the map $\sigma_R$ in the case of classical tori was automatic since the range was a commutative $C^\ast$-algebra. In our case, we have exploited the fact that the domain is $M_n(\mathbb{C})$.
\end{rmrk}

\begin{ppsn}\label{sigmaC1}
For every $T\in E_R\subseteq\mathcal B(H_R^0)$, we have
$$[D,\sigma_R(T)\otimes I_{2^{\lfloor{n/2}\rfloor}}]=(\sigma_R\otimes\mathrm{id}_{M_{2^{\lfloor{n/2}\rfloor}}(\mathbb C)})\big([D_R,\pi_R(T)]\big),$$
where $D_R=P_RDP_R$ is regarded as acting on $\mathcal B(H_R^0)\otimes M_{2^{\lfloor{n/2}\rfloor}}(\mathbb C)\cong \mathcal B(H_R)$ ($\pi_R(T)$ by \Eref{notn} is $T\otimes I_{2^{\lfloor{n/2}\rfloor}}$), and $\sigma_R\otimes\mathrm{id}_{M_{2^{\lfloor{n/2}\rfloor}}(\mathbb{C})}$ is the second-matrix amplification of the map $\sigma_R$. Thus, we have
$$\|[D,\sigma_R(T)\otimes I_{2^{\lfloor{n/2}\rfloor}}]\|\le\|[D_R,\pi_R(T)]\|.$$
\end{ppsn}
\begin{proof}
From \Lref{deltabound1}, we get $$[D_R,\pi_R(T)]=[D_R,T\otimes I_{2^{\lfloor{n/2}\rfloor}}]=\textstyle\sum_{j=1}^n[L_j,T]\otimes\sigma_j\in \mathcal B(H_R^0)\otimes M_{2^{\lfloor{n/2}\rfloor}}(\mathbb{C}),$$ and $[L_\mu,T]_{\lambda\nu}=(\lambda_\mu-\nu_\mu)T_{\lambda\nu}$ for $\mu=1,\ldots,n$. The second-amplification of the map $\sigma_R:\mathcal B(H_R^0)\to \mathcal{A}_{\Theta}$, i.e., $\sigma_R\otimes\mathrm{id}_{M_{2^{\lfloor{n/2}\rfloor}}(\mathbb C)}: \mathcal B(H_R^0)\otimes M_{2^{\lfloor{n/2}\rfloor}}(\mathbb{C})\to \mathcal{A}_{\Theta}\otimes M_{2^{\lfloor{n/2}\rfloor}}(\mathbb{C})$, gives us the following:
\begin{IEEEeqnarray*}{lCl}
(\sigma_R\otimes\mathrm{id}_{M_{2^{\lfloor{n/2}\rfloor}}(\mathbb C)})([D_R,T\otimes I_{2^{\lfloor{n/2}\rfloor}}]) &=&  (\sigma_R\otimes\mathrm{id}_{M_{2^{\lfloor{n/2}\rfloor}}})\big(\textstyle\sum_{j=1}^n[L_j,T]\otimes\sigma_j\big)\\
&=& \frac1{N_R}\textstyle\sum_{j=1}^n\sum_{\lambda,\nu}(\lambda_j-\nu_j)T_{\lambda\nu}e_\lambda e_\nu^*\otimes\sigma_j.
\end{IEEEeqnarray*}
As $\sigma_R(T)=\frac1{N_R}\sum_{\lambda,\nu}T_{\lambda\nu}e_\lambda e_\nu^*$, observe that $$\delta_\mu(e_\lambda e_\nu^*)=\delta_\mu(\overline{\omega(\lambda-\nu,\nu)}e_{\la-\nu})=(\la_\mu-\nu_\mu)\overline{\omega(\lambda-\nu,\nu)}(e_{\lambda-\nu})= (\la_\mu-\nu_\mu)e_\lambda e_\nu^*$$ for $\mu=1,\ldots,n$. Hence, by linearity we get
$$\delta_\mu(\sigma_R(T))=\frac1{N_R}\sum_{\lambda,\nu}T_{\lambda\nu}\,\delta_\mu(e_\lambda e_\nu^*)=\frac1{N_R}\sum_{\lambda,\nu}(\lambda_\mu-\nu_\mu)T_{\lambda\nu}\,e_\lambda e_\nu^*,\qquad \text{for } \mu=1,\ldots,n.$$
Therefore, $$(\sigma_R\otimes\mathrm{id}_{M_{2^{\lfloor{n/2}\rfloor}}(\mathbb C)})([D_R,T\otimes I_{2^{\lfloor{n/2}\rfloor}}])=\sum_{j=1}^n\delta_j(\sigma_R(T))\otimes \sigma_j.$$
On the other hand, $$[D,\sigma_R(T)\otimes I_{2^{\lfloor{n/2}\rfloor}}]=\sum_{j=1}^n\delta_j(\sigma_R(T))\otimes\sigma_j,$$ which proves the first equation.
 
For the second equation, as we know from \Lref{sigmaCP} that $\sigma_R:\mathcal B(H_R^0)\to \mathcal{A}_{\Theta}$ is completely contractive, we get that $\|[D,\sigma_R(T)\otimes I_{2^{\lfloor{n/2}\rfloor}}]\|=\|(\sigma_R\otimes\mathrm{id}_{M_{2^{\lfloor{n/2}\rfloor}}(\mathbb C)})([D_R,T\otimes I_{2^{\lfloor{n/2}\rfloor}}])\|\le\|[D_R,T\otimes I_{2^{\lfloor{n/2}\rfloor}}]\|$.
\end{proof}

\section{Gromov--Hausdorff convergence}
\subsection{Action of \texorpdfstring{$\mathbb{R}^n$}{R2} and Good kernel}
%\rm[$\mathbb{R}^n\overset{\alpha}{\curvearrowright} \mathcal{A}_{\Theta}$]

For the noncommutative torus $\mathcal{A}_{\Theta}$, we have the action of $\mathbb{T}^n=\mathbb{R}^n/(2\pi\bbz)^n$ on the $C^*$-algebra $\mathbb{T}_{\Theta}$ for which $\mathcal{A}_{\Theta}$ is the subalgebra of smooth elements. The covering map $\mathbb{R}^n\to \mathbb{T}^n$ given by $(s_1,\ldots,s_n)\mapsto(e^{is_1},\ldots,e^{is_n})$ lets us pull the $\mathbb{T}^n$-action back to $\mathbb{R}^n$, and we have the following fact:

\begin{lmma}\label{gauge}
For $x\in\mathbb{R}^n$, consider the map $\alpha_x(e_k):=e^{ik\cdot x}e_k$ for $e_k\in\mathcal{A}_{\Theta}$. Further, it extends to a $\ast$-automorphism of $\mathcal{A}_{\Theta}$ and $\|\alpha_x(a)\|=\|a\|$ for all $a\in \mathcal{A}_{\Theta}$. 
\end{lmma}

\begin{lmma}\label{gaugeC1}
For fixed $a\in \mathcal{A}_{\Theta}$ and $x=(x_1,\ldots,x_n)\in\mathbb R^n$, the curve $g(t):=\alpha_{tx}(a)$ is norm-$C^1$ (in fact, $C^\infty$), and
\[
g'(t)=i\textstyle\sum_{\mu=1}^nx_\mu\,\alpha_{tx}(\delta_\mu( a)).
\]
\end{lmma}
\begin{proof}
Write $g(t)=\sum_kf_k(t)$, where $f_k(t):=a_ke^{itk\cdot x}e_k\in \mathcal{A}_{\Theta}$. Each $f_k$ is $C^1$ with $f_k'(t)=i(k\cdot x)a_ke^{itk\cdot x}e_k$. Then $g(t_0)=\sum_kf_k(t_0)$ converges by using $\|e_k\|=1$ and rapid decay of $(a_k)$. For the derivative series, $\|f_k'(t)\|=|k\cdot x||a_k|\le\|x\|\|k\||a_k|$, and $\sum_k\|k\||a_k|<\infty$ by rapid decay, thus giving a bound independent of $t$. By Lemma \ref{P8}, $\sum_kf_k'(t)$ converges uniformly on $\mathbb{R}$. Therefore, we can apply Lemma \ref{P3} to get the following:
\begin{IEEEeqnarray*}{lCl}
g'(t) &=& \textstyle\sum_kf_k'(t)=\textstyle\sum_{k}a_k\,i(\sum_{\mu=1}^nk_\mu x_\mu) e^{itk\cdot x} e_k\\
&=& i\textstyle\sum_{k}\left(\sum_{\mu=1}^nx_\mu e^{itk\cdot x}k_\mu a_ke_k\right)\\
&=& i\Big(\textstyle\sum_{\mu=1}^nx_\mu\textstyle\sum_{k}e^{itk\cdot x} k_\mu a_ke_k\Big)\\
&=& i\Big(\textstyle\sum_{\mu=1}^nx_\mu\alpha_{tx}\big(\textstyle\sum_k k_\mu a_ke_k\big)\Big)\\
&=& i\Big(\textstyle\sum_{\mu}^nx_\mu\alpha_{tx}(\delta_\mu(a))\Big)\\
&=& i\textstyle\sum_{\mu=1}^{n}x_{\mu} \alpha_{tx}(\delta_{\mu}(a)),
\end{IEEEeqnarray*}
which concludes the proof.
\end{proof}

\begin{lmma}\label{FTC}
For all $a\in \mathcal{A}_\Theta$ and $x\in \mathbb{R}^n,$
\[
a-\alpha_x(a)=-i\int_0^1\textstyle\sum_{\mu=1}^nx_\mu\,\alpha_{tx}(\delta_\mu a)\,dt.
\]
\end{lmma}
\begin{proof}
By Lemma \ref{gaugeC1}, $g(t)=\alpha_{tx}(a)$ is $C^1$ and its derivative $g'(t)=i\sum_\mu x_\mu\alpha_{tx}(\delta_\mu a)$ is norm-continuous in $t$ (for each $\mu=1,\ldots,n$, the function $t\mapsto\alpha_{tx}(\delta_\mu( a))$ is given by a uniformly convergent Fourier series in $t$, by the same argument as Lemma \ref{gaugeC1} applied to $\delta_\mu a\in \mathcal{A}_{\Theta}$ in place of $a$). Applying Lemma \ref{P5}, we get $g(1)-g(0)=\int_0^1g'(t)\,dt$, which shows that $\alpha_x(a)-a=\int_0^1g'(t)\,dt$.
\end{proof}

We record a corollary of Lemma \ref{spinorcompression} that will be helpful.

\begin{crlre}\label{deltabound}
For $a\in \mathcal{A}_{\Theta}$ and $\mu=1,\ldots,n$, $\|\delta_\mu(a)\|\le\|[D,a]\|$.
\end{crlre}
\begin{proof}
Applying Lemma \ref{spinorcompression} with $K=L^2(\mathcal{A}_\Theta,\tau)$. We have $$X=[D,a]=\textstyle\sum_{\mu=1}^n\pi_\tau(\delta_\mu(a))\otimes\sigma_\mu.$$ This gives us $\|\delta_\mu(a)\|=\|\pi_\tau(\delta_\mu (a))\|\le\|[D,a]\|$.
\end{proof}

\begin{ppsn}\label{orbitbound}
Let $\tilde{x}\in \mathbb{T}^n=\mathbb{R}^n/(2\pi\bbz)^n$ and $x\in \mathbb{R}^n$ be any representative of $\tilde{x}$. For the $\mathbb{T}^n$-action $\alpha_{\tilde{x}}(e_k):=\alpha_{x}(e_k)$, where $\alpha_x$ is as in \Lref{gauge}, on $\mathcal{A}_{\Theta}$, we have 
\[
\|a-\alpha_{\tilde{x}}(a)\|\le\sqrt{n}\,\,{l}(\tilde{x})\,\|[D,a]\|,
\]
where ${l}(\tilde{x}):=\inf_{k\in \mathbb{Z}^n} \|x+2\pi k\|$.
\end{ppsn}
\begin{proof}
Since $\alpha_x=\alpha_{x+2\pi k}$ for any $k\in\bbz^n$, we have well-definedness of the $\mathbb{T}^n$-action. Now, applying \Lref{FTC}, \Lref{P6}, \Lref{gauge}, and \Yref{deltabound} in that order, we obtain the following:
\begin{IEEEeqnarray*}{lCl}
\|a-\alpha_{\tilde{x}}(a)\|=\|a-\alpha_x(a)\| &=& \Big\|-i\int_0^1\textstyle\sum_{\mu=1}^nx_\mu\,\alpha_{tx}(\delta_\mu a)\,dt\Big\|\\
&\le& \int_0^1\big\|\textstyle\sum_\mu x_\mu\alpha_{tx}(\delta_\mu a)\big\|dt\\
&\le& \textstyle\sum_\mu\int_0^1|x_\mu|\,\|\alpha_{tx}(\delta_\mu a)\|dt\\
&=& \textstyle\sum_\mu|x_\mu|\,\|\delta_\mu a\|\\
&\le& (|x_1|+\ldots+|x_n|)\|[D,a]\|\\
&\leq& \sqrt{n}\,\|x\|\,\|[D,a]\|.
\end{IEEEeqnarray*}
Taking infimum over all the representatives of $\tilde{x}$ in the last inequality proves the claim.
\end{proof}

Recall that $\Lambda_R:=\{k\in\mathbb Z^n:\|k\|\le R\}$ for $R\geq 0$.
\begin{dfn}
For any $\tilde{x}\in\mathbb T^n=\mathbb{R}^n/(2\pi\bbz)^n$, define
$$K_R(\tilde{x}):=\frac1{N_R}\Big|\sum_{\lambda\in\Lambda_R}e^{i\lambda\,\cdot\, x}\Big|^2,$$
where $x\in\mathbb{R}^n$ is any respresentative of $\tilde{x}$. Well-definedness follows from the definition of $\Lambda_R$.
\end{dfn}

\begin{lmma}\label{kernel}
We have $K_R(\tilde{x})\ge 0,\,\int_{\mathbb T^n}K_R(\tilde{x})\,d\tilde{x}=1$, and $K_R(\tilde{x})=\sum_{k}m_R(k)e^{ik\cdot x}$, where $m_R(k)$ is as defined in \Lref{mR(k)}.
\end{lmma}
\begin{proof}
To observe $\int_{\mathbb T^n}K_R(\tilde{x})\,d\tilde{x}=1$, use the normalized Haar probability measure on $\mathbb{T}^n=\mathbb{R}^n/(2\pi\bbz)^n$. We have $K_R(\tilde{x})=\frac1{N_R}\sum_{\lambda,\nu\in\Lambda_R}e^{i(\lambda-\nu)\cdot x}$, and grouping the same powers by $k=\lambda-\nu$ gives $K_R(\tilde{x})=\sum_km_R(k)e^{ik\cdot x}$.
\end{proof}

\begin{lmma}\label{average}
We have for any $a\in\mathcal{A}_{\Theta}$
\[
F_{m_R}(a)=\int_{\mathbb T^n}K_R(\tilde{x})\,\alpha_{\tilde{x}}(a)\,d\tilde{x},
\]
where $F_{m_R}$ is as defined in \Pref{eq:3.3}.
\end{lmma}
\begin{proof}
Let $S_k({\tilde{x}}):=\sum_{\|\ell\|\le k}a_\ell e^{i\ell\cdot x}e_\ell$. By the rapid decay of $(a_\ell)$ and $\|e_\ell\|=1$, we have $\|S_k({\tilde{x}})-\alpha_{\tilde{x}}(a)\|\le\sum_{\|\ell\|>n}|a_\ell|$, which tends to $0$ uniformly in ${\tilde{x}}$. Each $S_k$ is continuous map from $\mathbb T^n$ to $\mathcal{A}_{\Theta}$, and hence Bochner integrable by Lemma \ref{P4}. As $S_k\to \alpha_{\star}(a)$ uniformly and $K_R$ is bounded on the compact space $\mathbb T^n$, we have $K_R\cdot S_k\to K_R\cdot\alpha_{\star}(a)$ uniformly. By Corollary \ref{P7},
$$\int_{\mathbb T^n}K_R({\tilde{x}})S_k({\tilde{x}})\,d{\tilde{x}}\longrightarrow\int_{\mathbb T^n}K_R({\tilde{x}})\alpha_{\tilde{x}}(a)\,d{\tilde{x}}\quad\mathrm{as}\,\,k\to \infty.$$
On other hand, by \Lref{kernel} we get $$\int_{\mathbb{T}^n} K_R({\tilde{x}})S_k({\tilde{x}})\,d{\tilde{x}}=\sum_{\|\ell\|\le k}\left(\int_{\mathbb{T}^n} K_R({\tilde{x}})e^{i\ell\cdot x}d{\tilde{x}}\right)a_\ell e_\ell=\sum_{\|\ell\|\le k}m_R\left(-\ell\right)a_\ell e_\ell.$$ 
Since $F_{m_R}(a)=\sum_{\ell}m_R\left(\ell\right)a_\ell e_\ell$, using $m_R(-\ell)=m_R(\ell)$ and letting $k\to\infty$, we conclude the proof.
\end{proof}

The main results of this subsection are the following.
\begin{ppsn}\label{aside}
For any $a\in\mathcal{A}_{\Theta}$, we have
\[
a-\sigma_R\circ\rho_R(a)=\int_{\mathbb T^n}K_R({\tilde{x}})\big(a-\alpha_{\tilde{x}}(a)\big)\,d{\tilde{x}},
\] and
\[
\|a-\sigma_R\circ\rho_R(a)\|\le\gamma_R\,\|[D,a]\|,\qquad\gamma_R:=\sqrt{n}\int_{\mathbb T^n}K_R({\tilde{x}})l({\tilde{x}})\,d{\tilde{x}}.
\]
\end{ppsn}
\begin{proof}
By \Pref{compositions} and Lemma \ref{average}, we have $\sigma_R\rho_R(a)=F_{m_R}(a)=\int K_R({\tilde{x}})\alpha_{\tilde{x}}(a)\,d{\tilde{x}}$. Since $\int K_R=1$, we can take $a=\int K_R({\tilde{x}})a\,d{\tilde{x}}$, and this gives the first equation. For the second equation, using \Lref{P6} and \Pref{orbitbound} we have
\begin{equation*}
\|a-\sigma_R\rho_R(a)\|\le\int_{\mathbb T^n} K_R({\tilde{x}})\|a-\alpha_{\tilde{x}}(a)\|\,d{\tilde{x}}\le\sqrt{n}\Big(\int_{\mathbb T^n} K_R({\tilde{x}})l({\tilde{x}})\,d{\tilde{x}}\Big)\|[D,a]\|.
\end{equation*}
\end{proof}

\begin{ppsn}\label{concentration}
For every $\varepsilon>0$, let $B_\varepsilon(0)=\{{\tilde{x}}\in \mathbb{T}^n: l({\tilde{x}})<\varepsilon\}$, then $$\displaystyle\int_{\mathbb T^n\setminus B_\varepsilon(0)}K_R({\tilde{x}})\,d{\tilde{x}}\longrightarrow0$$ as $R\to\infty$. Consequently,
\[
\gamma_R:=\sqrt{n}\int_{\mathbb T^n}K_R({\tilde{x}})l({\tilde{x}})\,d{\tilde{x}}\longrightarrow 0\quad\mathrm{ as }\,\,R\to\infty.
\]
\end{ppsn}
\begin{proof}
Fix any $\varepsilon,\eta>0$. Choose a continuous function $f:\mathbb T^n\to[0,1]$ with $f(0)=0$ and $f\equiv1$ outside $B_\varepsilon(0)$ (e.g. $f(\tilde{x})=\min\left\{1, l(\tilde{x})/\varepsilon\right\}$). By the Stone--Weierstrass theorem, choose a trigonometric polynomial $p({\tilde{x}})=\sum_{\|k\|\le M}c_ke^{ik\cdot x}$ with $\|f-p\|_\infty<\eta$. Since $K_R\ge0$ and $\int K_R=1$, we have $\big|\int_{\mathbb T^n} K_R(f-p)\big|\le \|f-p\|_\infty \le\eta$. Also, $$\int_{\mathbb T^n} K_R\,p=\sum_{\|k\|\le M}c_km_R(k)\longrightarrow\sum_{\|k\|\le M}c_k=p(0)\quad\mathrm{ as }\,\,R\to \infty,$$ by Lemma \ref{mR(k)}. Observe that as $f(0)=0$, we have $|p(0)|=|f(0)-p(0)|\le\|f-p\|_\infty\le\eta$. Hence, $$\limsup_{R\to \infty}\left|\int_{\mathbb T^n} K_Rf\right|\le\limsup_{R\to \infty}\left(\left|\int_{\mathbb T^n} K_Rp\right|+\left|\int_{\mathbb T^n} K_R(f-p)\right|\right)\le2\eta.$$ Since $\eta>0$ was arbitrary, we get $\int_{\mathbb T^n} K_Rf\to0$. Since $f\ge\mathbf\chi_{\mathbb T^n\setminus B_\varepsilon(0)}$ and $K_R\ge0$, we have $\int_{\mathbb T^n\setminus B_\varepsilon(0)}K_R\le\int K_Rf$ which tends to $0$ as $R\to\infty$. This proves the first claim.
\smallskip

For the second part, let $M:=\sup_{\mathbb T^n}l({\tilde{x}})<\infty$. For $\varepsilon>0$, we get
$$\gamma_R=\sqrt{n}\int_{B_\varepsilon(0)}K_R({\tilde{x}})l({\tilde{x}})\,d{\tilde{x}}+\sqrt{n}\int_{\mathbb T^n\setminus B_\varepsilon(0)}K_R({\tilde{x}})l({\tilde{x}})\,d{\tilde{x}}\le\sqrt{n}\,\varepsilon+\sqrt{n}\,M\int_{\mathbb T^n\setminus B_\varepsilon(0)}K_R,$$
using $\int K_R=1$ and $l({\tilde{x}})< \epsilon$ for the first term. By Lemma \ref{concentration}, $\limsup_R\gamma_R\le\sqrt{n}\varepsilon$. As $\varepsilon>0$ is arbitrary, we get that $\gamma_R\to 0$ as $R\to\infty$.
\end{proof}

\begin{rmrk}\rm
Readers may visit subsection $3.3$ of \cite{LvS}. $K_R$ was the inverse Fourier transform of $\{m_{R}(k)\}_{k\in \mathbb{Z}^n}$; using this and few other facts from harmonic analysis, authors in \cite{LvS} proved that $K_R$ is a good kernel in Proposition $3.8$.
\end{rmrk}

\subsection{\texorpdfstring{$C^1$}{c1}-approximate order isomorphism and Convergence}

We begin by fixing the following notation.
%\rm[$\mathbb{R}^n\overset{\beta}{\curvearrowright} M_{N_R}(\mathbb{C})$]
\begin{notation}\label{notn2}
For $x=(x_1,\ldots,x_n)\in\mathbb R^n$, let $$W_x:=e^{i(\sum_{\mu=1}^nx_\mu L_\mu)}=\sum_{m\geq 0}\frac{\big(i(\sum_{\mu=1}^nx_\mu L_\mu)\big)^m}{m!}\in \mathcal{B}(H_R^0)$$
where $L_\mu$'s are as defined in \Eref{to refer in 4.2}. Using the fact that $L_\mu$'s commute with each other, we get that $W_xe_\lambda=e^{i\lambda\cdot x}e_\lambda$. We also see  that $W_x=W_{x+2\pi k}$, because  $W_xe_{\la}=e^{i\la\cdot x}e_\la=e^{i\la\cdot (x+2\pi k)}e_\la=W_{x+2\pi k}e_\la$, for all $\la\in \Lambda_R$.
\end{notation}
\medskip

Define $\beta_x(T):=W_xTW_x^*$ for $T\in \mathcal{B}(H_R^0)$ and $x\in\mathbb{R}^n$. Clearly $\|\beta_x(T)\|=\|T\|$ and $\beta_x(\rho_R(a))=\rho_R(\alpha_x(a))$, for all $a\in \mathcal{A}_{\Theta}$ and $x\in \mathbb{R}^n$.

\begin{lmma}\label{betacommutator}
For every $\,T\in \mathcal{B}(H_R^0)$ and $x\in \mathbb{R}^n$, the map $\mathbb{R}\to \mathcal{B}(H_R^0)$ given by $ t\longmapsto\beta_{tx}(T)$
is norm-$C^1$ and
\begin{equation}
\frac{d}{dt}\beta_{tx}(T)
=i\,\beta_{tx}\bigl(\textstyle\sum_{\mu=1}^nx_\mu[L_\mu,T]\bigr).
\label{eq:7.3}
\end{equation}
\end{lmma}
\begin{proof}
Define $x\cdot L:=\sum_{\mu=1}^nx_\mu L_\mu$, which is self-adjoint, and observe that $W_{tx}=e^{it(x\cdot L)}$ commutes with $x\cdot L$. The product rule gives us the following:
\begin{IEEEeqnarray*}{lCl}
\frac{d}{dt}\big(W_{tx}TW_{tx}^*\big) &=& i(x\cdot L)W_{tx}TW_{tx}^*-W_{tx}TW_{tx}^*\,i(x\cdot L)\\
&=& iW_{tx}\big[x\cdot L\,,\,T\big]W_{tx}^*\\
&=& i\beta_{tx}\big([x\cdot L\,,\,T]\big),
\end{IEEEeqnarray*}
which concludes the proof.
\end{proof}

For ${\tilde{x}}\in \mathbb{T}^n$, with a representative $x\in \mathbb{R}^n$, define the map $\beta_{\tilde{x}}(T):=\beta_x(T)$, for all $T\in \mathcal{B}(H_R^0)$. Then, we have an well-defined action $\beta$ of $\mathbb{T}^n$ on $\mathcal{B}(H_R^0)$. Observe that  $\|\beta_{\tilde{x}}(T)\|=\|T\|$ and $\beta_{\tilde{x}}(\rho_R(a))=\rho_R(\alpha_{\tilde{x}}(a))$, for all $a\in \mathcal{A}_{\Theta}$ and ${\tilde{x}}\in \mathbb{T}^n$.

\begin{ppsn}\label{Tside}
For all $T\in E_R\subseteq\mathcal{B}(H_R^0)$ and $x\in \mathbb{R}^n$, we have
\[
T-\beta_x(T)=-i\int_0^1\beta_{tx}\big(\textstyle\sum_{\mu=1}^nx_\mu[L_\mu,T]\big)\,dt\,,\quad\|T-\beta_x(T)\|\le\sqrt{n}\,\|x\|\,\|[D_R,T\otimes I_{2^{\lfloor{n/2}\rfloor}}]\|.
\]
Consequently, for all $\tilde{x}\in \mathbb{T}^n,$
\[
\|T-\beta_{\tilde{x}}(T)\|\leq \sqrt{n}\  l({\tilde{x}}) \|[D_R,T\otimes I_{2^{\lfloor{n/2}\rfloor}}]\|.
\]
\end{ppsn}
\begin{proof}
By Lemma \ref{betacommutator}, the map $t\longmapsto\beta_{tx}(T)$ is norm-$C^1$. Hence, the fundamental theorem of calculus for Bochner
integrals gives us
$$\beta_x(T)-T=\int_0^1\frac{d}{dt}\beta_{tx}(T)\,dt.$$
Thus,
$$T-\beta_x(T)=-i\int_0^1\beta_{tx}\bigl(\textstyle\sum_{\mu=1}^nx_\mu[L_\mu,T]\bigr)\,dt.$$
As $\|\beta_x(T)\|=\|T\|$, using \Lref{P6} we get the following:
\begin{IEEEeqnarray*}{lCl}
\|T-\beta_x(T)\| &\le& \int_0^1\big\|\textstyle\sum_{\mu=1}^nx_\mu[L_\mu,T]\big\|\,dt\\
&\le& \sum_{\mu=1}^n|x_\mu|\|[L_\mu,T]\|.
\end{IEEEeqnarray*}
By \Lref{deltabound1}, we have $\|[L_\mu,T]\|\le\|[D_R,T\otimes I_{2^{\lfloor{n/2}\rfloor}}]\|$ for $\mu=1,\ldots,n.$
Hence,
\begin{IEEEeqnarray*}{lCl}
\|T-\beta_x(T)\| &\le& (\textstyle\sum_{\mu=1}^n|x_\mu|)\,\|[D_R,T\otimes I_{2^{\lfloor{n/2}\rfloor}}]\|\\
&\le& \sqrt{n}\,\|x\|\,\|[D_R,T\otimes I_{2^{\lfloor{n/2}\rfloor}}]\|.
\end{IEEEeqnarray*}
As $\beta_{\tilde{x}}(T):= \beta_x(T)$, taking infimum over all the representatives of $\tilde{x}$ gives us the desired inequality.
\end{proof}

\begin{lmma}\label{betax}
For $T\in E_R\subseteq\mathcal{B}(H_R^0)$ and $\tilde{x}\in \mathbb{T}^n,$ we have
$(\beta_{\tilde{x}}(T))_{\lambda\nu}=e^{i(\lambda-\nu)\cdot x}T_{\lambda\nu}$, and
\[
\int_{\mathbb T^n}K_R({\tilde{x}})\beta_{\tilde{x}}(T)\,d{\tilde{x}}=S_{m_R}(T),
\]
where $S_{m_R}(T)$ is as defined in \eqref{eq:3.4}.
\end{lmma}
\begin{proof}
The first equation is immediate from $W_xe_\lambda=e^{i\lambda\cdot x}e_\lambda$. Now, observe that $\rho_R\circ F_{m_R}=S_{m_R}\circ\rho_R$, and let $T=\rho_R(a)$ for some $a\in \mathcal{A}_{\Theta}$. Then, by \Lref{average} we get
\begin{IEEEeqnarray*}{lCl}
S_{m_R}(T) &=& \rho_R(F_{m_R}(a))\\
&=& \rho_R\left(\int_{\mathbb{T}^n}K_R({\tilde{x}})\alpha_{\tilde{x}}(a)\,d{\tilde{x}}\right).
\end{IEEEeqnarray*}
Since $\rho_R$ is a bounded linear map, using Theorem \ref{P7''} we get the following:
$$ S_{m_R}(T)=\int_{\mathbb{T}^n}K_R({\tilde{x}})\rho_R(\alpha_{\tilde{x}}(a))\,d{\tilde{x}}=\int_{\mathbb{T}^n}K_R({\tilde{x}})\beta_{\tilde{x}}(T)\,d{\tilde{x}}.$$
\end{proof}

\begin{ppsn}\label{tside}
For $T\in E_R\subseteq\mathcal{B}(H_R^0)$, we have
\[
T-\rho_R\sigma_R(T)=\int_{\mathbb T^n}K_R({\tilde{x}})\big(T-\beta_{\tilde{x}}(T)\big)\,d{\tilde{x}}\,,\quad \|T-\rho_R\sigma_R(T)\|\le\gamma_R\|[D_R,T\otimes I_2]\|.
\]
\end{ppsn}
\begin{proof}
By Lemma \ref{betax} and \Pref{compositions}, we have $\int_{\mathbb T^n} K_R({\tilde{x}})\beta_{\tilde{x}}(T)\,d{\tilde{x}}=S_{m_R}(T)=\rho_R\sigma_R(T)$. Subtracting from $T=\int_{\mathbb T^n} K_R({\tilde{x}})T\,d{\tilde{x}}$ gives the first equation. For the second equation, by \Lref{P6} we have 
\begin{IEEEeqnarray*}{lCl}
\|T-\rho_R\sigma_R(T)\| &\le& \int_{\mathbb T^n} K_R({\tilde{x}})\|T-\beta_{\tilde{x}}(T)\|\,d{\tilde{x}}\\
&\le& \sqrt{n}\Big(\int_{\mathbb T^n} K_R({\tilde{x}})l({\tilde{x}})\,d{\tilde{x}}\Big)\|[D_R,T\otimes I_{2^{\lfloor{n/2}\rfloor}}]\|=\gamma_R\|[D_R,T\otimes I_{2^{\lfloor{n/2}\rfloor}}]\|,
\end{IEEEeqnarray*}
where the second inequality follows from \Pref{Tside}.
\end{proof}

\begin{rmrk}\rm
In the classical case of torus, the fact that $\rho_R\circ\sigma_R$ is an approximate identity of $E_R$ was obtained via a variant of the result of Bożejko and Fendler, see \cite[Lemma 3.5]{LvS}. The argument in that lemma heavily relies on the commutativity of $C(\mathbb T^d)$ to pass from the completely bounded norm to the usual operator norm, and hence does not apply in our noncommutative setting.
\end{rmrk}

The main result of the paper is the following.
\begin{thm}\label{main}
The state spaces $(\mathcal{S}(E_R),d_{E_R})$ converge to $(\mathcal{S}(\mathcal{A}_{\Theta}),d_{\mathcal{A}_{\Theta}})$ in the Gromov--Hausdorff distance as $R\to\infty$.
\end{thm}
\begin{proof}
We have the maps $\rho_R(a)=P_R^0\pi_\tau(a)P_R^0$ and $\sigma_R(T)=\frac1{N_R}\sum_{\lambda,\nu\in\Lambda_R}T_{\lambda\nu}e_\lambda e_\nu^*$ from Section $3$. By Lemma \ref{rho} and \ref{sigmaCP}, we have unitality, complete positivity, and norm-contractivity of the maps $\rho_R$ and $\sigma_R$, respectively, for each $R\geq 0$. The Lipschitz-contractivity of $\rho_R$ follows from \Pref{rho}, and that of $\sigma_R$ follows from \Pref{sigmaC1}. Finally, \Pref{aside} and \Pref{tside} with the common $\gamma_R\to0$ (see Proposition \ref{concentration}) establishes $(\rho_R,\sigma_R)$ as a $C^1$-approximate order isomorphism between $(\mathcal{A}_{\Theta},H,D)$ and the spectrally truncated spaces $\{(E_R,H_R,D_R)\}_{R\ge 0}$. Consequently, by \Tref{Thmofsuijlekom} the result follows.
\end{proof}

As we have mentioned earlier, the quantum Gromov--Hausdorff convergence of the truncated operator systems to the noncommutative torus follows from the above result, applying Proposition $2.2.4$ in \cite{KK1}.

Our entire analysis survives in the case of $\Theta=0$. Consequently, we recover the main result of \cite{LvS} as a special case.
\begin{crlre}
The state spaces of the spectral truncations for tori converges to the state space of tori in the Gromov--Hausdorf distance.
\end{crlre}

\section{Structure analysis: \texorpdfstring{$C^*$}{C*}-envelope and propagation number}

The $C^*$-envelope and propagation number dertermines further sturctures of the truncated operator systems. We do not recall their definitions and properties here, instead we refer to \cite{CS1,F,LvS} for all the details.

This section mostly follows from \cite[Section 4.1]{LvS}, except for a few necessary changes due to nontrivial cocycles, and so we make it brief.

Recall $L(s)$, for $s\in \mathbb{Z}^n$, from Lemma \ref{mR(k)}. For $k,\ell\in \Lambda_R$, let $E_{k,\ell}\in \mathcal{B}(H_R^0)\cong M_{N_R}(\mathbb{C})$ be the standard matrix unit. Let $p\in \mathbb{Z}^{n}$, and define $T_{p}\in \mathcal{B}(H_R^0)\cong M_{N_R}(\mathbb{C})$ by $T_{p}:=\rho_{R}(e_{-p})$. Then \Lref{matrixcoeff} implies that $(T_p)_{\la\mu}=(e_{-p})_{\la-\mu}w(\la-\mu,\mu)$, where $(e_{-p})_{\la-\mu}$ is the $(\la-\mu)^{\text{th}}$-coefficient of the Fourier expansion of $e_{-p}$, which is $1$ if and only if $\la-\mu=-p$. Hence $(T_p)_{\la\mu}$ is non-zero if and only if $\la-\mu=-p$, and in that case $(T_p)_{\la\mu}=w(-p,\mu)$. This proves the following:
\begin{equation}\label{eq1 of envelope}
T_p=\sum_{s\in L(p)} w(-p,s)E_{s-p,s}.
\end{equation}
From the above reasoning, it is also evident that $T_p=0$ if $p\notin \Lambda_R-\Lambda_R$, because $p\notin \Lambda_R-\Lambda_R$ if and only if $L(p)=\emptyset$. The support of the matrix $T_p$ is $\{(s-p,s):s\in L(p)\}$. Further, it is easy to see that $\{T_p:p\in \Lambda_R-\Lambda_R\}$ is a basis of $E_R$. Finally,
$$
T_{(0,0,\dots,0)}=\sum_{s\in \Lambda_R} E_{s,s}=\mathrm{id}_{H_R^0}.
$$

\begin{lmma}\label{lemma 1 of envelope}
Let $p,q\in \Lambda_R-\Lambda_R$. Then we have
$$T_pT_q=w(-p,-q)\sum_{s\in L(p+q)\,\cap\, L(q)}w(-(p+q),s)E_{s-p-q,s}.$$
\end{lmma}
\begin{proof}
Let $T_p=\sum_{m\in L(p)} w(-p,m)E_{m-p,m}$ and $T_q=\sum_{s\in L(q)} w(-q,s)E_{s-q,s}$. Then the proof goes verbatim from \cite{LvS}, but this time we have extra factors, i.e.,
$$T_pT_q=\sum_{s\in L(q)\,\cap\, L(p+q)} w(-p,s-q)w(-q,s)E_{s-p-q,s}.$$
Using the cocycle identities, we get that
$$w(-p,s-q)w(-q,s)=w(-p,s)w(-p,-q)w(-q,s)=w(-p,-q)w(-(p+q),s)$$ which completes the proof.
\end{proof}

As a special case, for $\ell,k\in \Lambda_R-\Lambda_R$ such that $\ell+k\in \Lambda_R-\Lambda_R$, we get, for $p=-k$ and $q=\ell+k$,
\begin{IEEEeqnarray}{lCl}\label{eq2 of  envelope}
T_{-k}T_{\ell+k}=w(k,-(\ell+k))\sum_{s\in L(\ell)\,\cap\, L(\ell+k)}w(-\ell,s)E_{s-\ell,s}.
\end{IEEEeqnarray}
Now define, for $r\geq 0$, $\sk_r:= \mathrm{co}(\Lambda_r)\subseteq \mathbb{R}^n$ (the closed convex hull of $\Lambda_{r}$). Clearly, $\sk_r\subseteq \overline{B}(0,r)$, $\sk_r\cap \mathbb{Z}^n= \Lambda_r$, and $\mathrm{Ext}(\sk_r)\subseteq \Lambda_r$.

\begin{lmma}\label{lemma 2 of envelope}
Fix $p,q\in \Lambda_R$, and define $\ell:=q-p\in\Lambda_R-\Lambda_R$. Then there exists $m\in\mathrm{Ext}(\sk_R)$ such that
$$\sk_{\|q\|}\cap (\sk_R -m+q)=\{q\}.$$ Put $k:=p-m\in \Lambda_R-\Lambda_R$. Then $\ell+k\in \Lambda_R-\Lambda_R$, and $L(\ell)\cap L(\ell+k)\cap \sk_{\|q\|}=\{q\}$. Consequently, $$s\in L(\ell)\cap L(\ell+k),\,s\neq q\implies \|s\|>\|q\|.$$
\end{lmma}
\begin{proof}
This follows verbatim from \cite{LvS}.
\end{proof}

With the fixed $p,q\in \Lambda_R$ and $\ell,k$ as in Lemma \ref{lemma 2 of envelope}, define $$c_{k,\ell}(s):=w(k,-(\ell+k))w(-\ell,s)\in \mathbb{S}^1\text{ for every } s.$$ Then, from \eqref{eq2 of  envelope}, we get that
\begin{IEEEeqnarray*}{lCl}
T_{-k}T_{\ell+k} &=& \sum_{s\in L(\ell)\,\cap\, L(\ell+k)}c_{k,\ell}(s)E_{s-\ell,s}\\
&=& \sum_{s\in L(\ell)\,\cap\, L(\ell+k)\cap \sk_{\|q\|}}c_{k,\ell}(s)E_{s-\ell,s}+\sum_{s\in L(\ell)\,\cap\, L(\ell+k)\cap \sk_{\|q\|}^{c}}c_{k,\ell}(s)E_{s-\ell,s}\\
&=& c_{k,\ell}(q)E_{p,q}+\sum_{\substack{s\in L(\ell)\,\cap\, L(\ell+k) \\ \|s\|>\|q\|}}c_{k,\ell}(s)E_{s-\ell,s}.
\end{IEEEeqnarray*}
Consequently,
\begin{equation}\label{eq3 for envelope}
E_{p,q}=c_{k,\ell}(q)^{-1} T_{-k}T_{\ell+k}-\sum_{\substack{s\in L(\ell)\,\cap\, L(\ell+k) \\ \|s\|>\|q\|}}\frac{c_{k,\ell}(s)}{c_{k,\ell}(q)}E_{s-\ell,s}.
\end{equation}
Furthermore, repeating this process for each $E_{s-\ell,s}$ appearing in the sum leads, in finitely many iterations, to a finite linear combination of products of the form \eqref{eq2 of  envelope} for $E_{p,q}$.

\begin{thm}\label{main result 2}
For each $R \geq 1$, the $C^*$-envelope and the propagation number of the truncated space $E_R$ is given by $C^*_e(E_R)=\mathcal{B}(H_R^0)\cong M_{N_R}(\mathbb{C})$ and $\mathrm{prop}(E_R)=2.$
\end{thm}
\begin{proof}
That $\mathcal{B}(H_R^0)$ is a $C^*$-extension of $E_R$ is immediate now, as by the above argument, any standard matrix unit $E_{p,q}$ lies in the $\mathrm{span}\{T_{-k}T_{\ell+k}:\ell,k,\ell+k\in\Lambda_R-\Lambda_R\}.$ Further, we get that the $C^*$-envelope of $E_R$ is also $\mathcal{B}(H_R^0)\cong M_{N_R}(\mathbb{C})$ by the simpleness, and the above argument also shows that $\mathrm{prop}(E_R)\leq 2$. Now, to show that $\mathrm{prop}(E_R)>1$, observe that for any $T=\rho\left(\sum_{\la\in \mathbb{Z}^n} a_\la e_\la\right)\in E_R$, the entries on the leading diagonal are all the same, i.e., $T_{\la\la}=a_0$, for all $\la\in \Lambda_R$. Thus $E_{0,0}\notin E_R$, which concludes the proof.
\end{proof}

This shows that the deeper structure of the truncated operator systems for tori survives deformations. Despite our efforts, the structure of dual for the truncated operator systems remains largely unexplored, much as in the classical case of tori.

\bigskip

\noindent{\sc Vibhor Bhatt} (\texttt{vibhorbhatt25@iitk.ac.in})\\
{\footnotesize Department of Mathematics,
Indian Institute of Technology, Kanpur 208016, India}
\vspace*{.5cm}

\noindent{\sc Satyajit Guin} (\texttt{satyamath@gmail.com; sguin@iitk.ac.in})\\
{\footnotesize Department of Mathematics,
Indian Institute of Technology, Kanpur 208016, India}
\vspace*{.5cm}

\noindent{\sc Palash Sharma} (\texttt{palashsharma8765434124@gmail.com; ms22001@iisermohali.ac.in})\\
{\footnotesize Department of Mathematical Sciences,
Indian Institute of Science Education and Research, Mohali,\\ Knowledge City, Sector 81, Punjab 140 306, India}


\begin{thebibliography}{10}\footnotesize

\bibitem{AKK}
Aguilar, K.; Kaad, J.; Kyed, D.:
\newblock The Podle\'s spheres converge to the sphere.
\newblock {\em Comm. Math. Phys.} 392 (2022), no. 3, 1029--1061.

\bibitem{A1}
Arveson, W.~B.:
\newblock Subalgebras of $C^*$-algebras.
\newblock {\em Acta Math.}, 123 (1969), 141--224.

\bibitem{BGB}
Bhatt, V.; Guin, S.; Saurabh, B.:
\newblock Quantum Gromov--Hausdorff Convergence for Extensions of $C^*$-Algebras. arXiv:2607.06158.

\bibitem{CG}
Chakraborty, P.~S.; Guin, S.:
\newblock Equivalence of two approaches to Yang-Mills on noncommutative torus.
\newblock {\em J. Noncommut. Geom.} 9 (2015), no. 2, 447--471.

\bibitem{C1}
Connes, A.:
\newblock Compact metric spaces, Fredholm modules and hyperfiniteness.
\newblock {\em Ergodic Theory and Dynamical Systems}, 9 (1989), no. 2, 207--220.

\bibitem{C2}
Connes, A.:
\newblock Noncommutative Geometry.
\newblock Academic Press, San Diego, 1994.

\bibitem{CS1}
Connes, A.; van Suijlekom, W.~D.:
\newblock Spectral Truncations in Noncommutative Geometry and Operator Systems. \newblock {\em Comm. Math. Phys.} 383 (2021), 2021--2067.

\bibitem{CS2}
Connes, A; van Suijlekom, W.~D.:
\newblock Tolerance relations and operator systems.
\newblock {\em Acta Sci. Math. (Szeged)} 88 (2022), no. 1-2, 101--129.

\bibitem{D}
Dieudonn\'e, J.:
\newblock Foundations of modern analysis.
\newblock Pure Appl. Math., Vol. X {\em Academic Press, New York-London}, 1960.

\bibitem{DU}
Diestel, J.; Uhl, J. J., Jr.:
\newblock Vector Measures.
\newblock Math. Surveys, No. 15 {\em American Mathematical Society, Providence, RI}, 1977.

\bibitem{F}
Farenick, D.:
\newblock The operator system of Toeplitz matrices.
\newblock {\em Trans. Amer. Math. Soc. Ser. B} 8 (2021), 999--1023.

\bibitem{GS}
Gaudillot-Estrada, Y.; van Suijlekom, W. D.:
\newblock Convergence of spectral truncations for compact metric groups.
\newblock {\em Int. Math. Res. Not. IMRN} 2025, no. 13, Paper No. rnaf197.

\bibitem{GVF}
Gracia-Bond\'ia, J.~M.; V\'arilly, J.~C.; Figueroa, H.:
\newblock Elements of noncommutative geometry.
\newblock Birkhäuser Adv. Texts Basler Lehrbücher [Birkhäuser Advanced Texts: Basel Textbooks] {\em Birkhäuser Boston, Inc., Boston, MA}, 2001.

\bibitem{H}
Hamana, M.;
\newblock Injective envelopes of operator systems.
\newblock {\em Publ. Res. Inst. Math. Sci.} 15 (1979), no. 3, 773--785.

\bibitem{KK}
Kaad, J.; Kyed, D.:
\newblock A comparison of two quantum distances.
\newblock {\em Math. Phys. Anal. Geom.} 26 (2023), no. 1, Paper No. 8.

\bibitem{KK1}
Kaad, J.; Kyed, D.:
\newblock The quantum metric structure of quantum  $SU(2)$. 
\newblock {\em Mem. Eur. Math. Soc.}, 18 EMS Press, Berlin, 2025.

\bibitem{L1}
Latrémolière, F.:
\newblock The quantum Gromov--Hausdorff propinquity.
\newblock {\em Trans. Amer. Math. Soc.}, 368 (2016), no. 1, 365--411.

\bibitem{L2}
Latrémolière, F.:
\newblock The Gromov--Hausdorff propinquity for metric spectral triples.
\newblock {\em Adv. Math.} 404 (2022), part A, Paper No. 108393.

\bibitem{L}
Leimbach, M.:
\newblock Convergence of Peter–Weyl truncations of compact quantum groups.
\newblock {\em J. Noncommut. Geom.}, to appear.

\bibitem{LvS}
Leimbach, M.; van Suijlekom, W.~D.: Gromov--Hausdorff convergence of spectral truncations for tori.
\newblock {\em Adv. Math.} 439 (2024), 109496.

\bibitem{Li}
Li, H.:
\newblock $\theta$-deformations as compact quantum metric spaces.
\newblock {\em Comm. Math. Phys.} 256 (2005), no. 1, 213--238.

\bibitem{P}
Paulsen, V. I.:
\newblock Completely Bounded Maps and Operator Algebras. 
\newblock Cambridge Studies in Advanced Mathematics, Cambridge University Press, 2003.

\bibitem{Rie0}
Rieffel, M.~A.
\newblock Noncommutative tori--a case study of noncommutative differentiable manifolds.
\newblock Contemp. Math., 105. {American Mathematical Society, Providence, RI}, 1990.

\bibitem{Rie1}
Rieﬀel, M.~A.:
\newblock Metrics on states from actions of compact groups.
\newblock {\em Doc. Math.}, 3 (1998), 215--229.

\bibitem{Rie2}
Rieﬀel, M.~A.:
\newblock Metrics on state spaces.
\newblock {\em Doc. Math.}, 4 (1999), 559--600.

\bibitem{Rie3}
Rieﬀel, M.~A.:
\newblock Gromov--Hausdorff distance for quantum metric spaces.
\newblock {\em Mem. Amer. Math. Soc.}, 168 (2004), no. 796, 1--65.

\bibitem{Rie4}
Rieffel, M.~A.:
\newblock Convergence of Fourier truncations for compact quantum groups and finitely generated groups.
\newblock {\em J. Geom. Phys.}, 192 (2023), Paper No. 104921.

\bibitem{T}
Toyota, R.:
\newblock Quantum Gromov--Hausdorff convergence of spectral truncations for groups with polynomial growth.
\newblock arXiv:2309.13469.

\bibitem{vS}
van Suijlekom, W. D.:
\newblock Gromov--Hausdorff convergence of state spaces for spectral truncations.
\newblock {\em J. Geom. Phys.}, 162, 2021.

\end{thebibliography}
\end{document}